\documentclass[10pt]{amsart}
      \usepackage[mathscr]{eucal}
      \usepackage{amsmath,amsfonts}

\usepackage{mathtools}
\usepackage{mathrsfs}
\usepackage{stackengine,scalerel}

\newcommand\obullet[1]{\ThisStyle{\ensurestackMath{%
  \stackon[1pt]{\SavedStyle#1}{\SavedStyle\kern.6\LMpt\bullet}}}}
\newcommand\ocirc[1]{\ThisStyle{\ensurestackMath{%
  \stackon[1pt]{\SavedStyle#1}{\SavedStyle\kern.6\LMpt\circ}}}}

      \parskip\smallskipamount
          \newtheorem{theorem}{Theorem}[section]
      \newtheorem{definition}[theorem]{Definition}
      \newtheorem{proposition}[theorem]{Proposition}
      \newtheorem{corollary}[theorem]{Corollary}
      
      \newtheorem{example}[theorem]{Example}
      
      \newtheorem{remark}[theorem]{Remark}
      \makeatletter
      \@addtoreset{equation}{section}
      \makeatother

      \newcommand{\CC}{{\mathbb C}}
      \newcommand{\NN}{{\mathbb N}}

      \newcommand{\RR}{{\mathbb R}}
      \newcommand{\FF}{{\mathbb F}}

\DeclareMathOperator{\Span}{span}

      \newcommand{\cA}{{\mathcal A}}
      \newcommand{\cB}{{\mathcal B}}
      \newcommand{\cC}{{\mathcal C}}
      \newcommand{\cD}{{\mathcal D}}
      \newcommand{\cE}{{\mathcal E}}
      \newcommand{\cF}{{\mathcal F}}
      \newcommand{\cG}{{\mathcal G}}
      \newcommand{\cH}{{\mathcal H}}
      
      \newcommand{\cJ}{{\mathcal J}}
      \newcommand{\cK}{{\mathcal K}}
      \newcommand{\cL}{{\mathcal L}}
      \newcommand{\cM}{{\mathcal M}}
       
      \newcommand{\cQ}{{\mathcal Q}}
      \newcommand{\cN}{{\mathcal N}}
      \newcommand{\cP}{{\mathcal P}}
      
      \newcommand{\cS}{{\mathcal S}}

      \newcommand{\cY}{{\mathcal Y}}

      \newcommand{\rank}{\hbox{\rm{rank}}\,}

      \newdimen\expt
      \def\boxit#1{\setbox0\hbox{$\displaystyle{#1}$}
            \hbox{\lower.4\expt
       \hbox{\lower3\expt\hbox{\lower\dp0
            \hbox{\vbox{\hrule height.4\expt
       \hbox{\vrule width.4\expt\hskip3\expt
            \vbox{\vskip3\expt\box0\vskip2\expt}%
       \hskip3\expt\vrule width.4\expt}\hrule height.4\expt}}}}}}
       
\begin{document}
     
       \pagestyle{myheadings}
      \markboth{ Gelu Popescu}{ Noncommutative domains, universal operator models, and operator algebras, II }
      %\pagestyle{plain}
      %\begin{flushright}
       % \it Date of this draft: \today
      %\end{flushright}
      %\bigskip

      \title [ Noncommutative domains, universal operator models, and operator algebras, II  ]
      {   Noncommutative domains, universal operator models, and operator algebras, II    }
        \author{Gelu Popescu}
     % \date{\today}
\date{March 2, 2023}
     \thanks{Research supported in part by  NSF grant DMS 1500922}
      \address{Department of Mathematics, The University of Texas
      at San Antonio \\ San Antonio, TX 78249, USA}
      \email{\tt gelu.popescu@utsa.edu}

\begin{abstract}    In a recent  paper, we introduced and studied the class of admissible noncommutative domains  $\cD_{g^{-1}}(\cH)$  in $B(\cH)^n$ associated with admissible free holomorphic functions $g$ in noncommutative indeterminates $Z_1,\ldots, Z_n$. Each such a domain  admits a universal model  ${\bf W}:=(W_1,\ldots, W_n)$ of weighted left creation operators acting on the full Fock space  with $n$ generators. In the present paper, we continue the study of these domains and their universal models in connection  with the Hardy algebras and the $C^*$-algebras they generate.  We obtain  a Beurling type characterization of  the invariant subspaces of the universal model   ${\bf W}:=(W_1,\ldots, W_n)$ and develop a dilation theory for the elements of  the noncommutative  domain
 $\cD_{g^{-1}}(\cH)$. We also obtain results concerning the commutant lifting  and Toeplitz-corrona  in our setting as as well as some results on the  boundary property for universal models.
\end{abstract}

      \maketitle

\section*{Contents}
{\it

\quad Introduction
\begin{enumerate}
   \item[1.]  Universal operator  models associated with  admissible free holomorphic functions
  \item[2.]     Invariant subspaces   for  universal  operator models
   \item[3.]     Minimal dilations for pure $n$-tuples  of operators and dilation index
    \item[4.]     Dilation theory on  noncommutative domains
   \item[5.]      Commutant lifting  and Toeplitz-corona theorems
   \item[6.]     Boundary property for universal models 
   \end{enumerate}
    
    \quad References

}

\section*{Introduction}

A free analogue of Sz.-Nagy--Foia\c s   \cite{SzFBK-book}, \cite{S}  theory of contractions  
 for  the closed unit ball
$$
[B(\cH)^n]_1:=\{(X_1,\ldots, X_n)\in B(\cH)^n:\ X_1 X_1^*+\cdots +X_nX_n^*\leq I\}
$$  
  has been developed over the last thirty years (see \cite{F}, \cite{B}, \cite{Po-isometric}, \cite{Po-charact}, \cite{Po-von}, \cite{Po-funct}, \cite{Po-analytic},  \cite{Po-poisson}, \cite{APo1}, \cite{APo2}, \cite{Po-curvature},  \cite{Po-Nehari}, \cite{Po-entropy}, \cite{Po-varieties}, \cite{Po-unitary}, \cite{Po-automorphisms},   \cite{BV}, \cite{DP1}, \cite{DP2}, \cite{DKP}, \cite{SSS1}, \cite{SSS2} and the reference there in), where $B(\cH)$ is  the algebra of all bounded linear operators on a Hilbert space $\cH$.  In the commutative setting (i.e. $X_iX_j=X_jX_i$) 
   the corresponding commutative ball was studied by  Drury \cite{Dru}, Arveson \cite{Arv3-acta}, Bhattacharyya-Eschmeier-Sarkar \cite{BES}, Timotin \cite{T},  and by the author \cite{Po-poisson}, \cite{Po-varieties}. Some of these results we extended to more general commutative domains by   Agler  in \cite{Ag1}, \cite{Ag2},  Athavale \cite{At}, M\" uller \cite{M}, M\" uller-Vasilescu \cite{MV},
   Vasilescu \cite{Va},  and Curto-Vasilescu \cite{CV1},    S. Pott \cite{Pot},
     Olofsson \cite{O1}, \cite{O2},   Ball and Bolotnikov \cite{BB}, Sarkar \cite{S1}, \cite{S2}, and others.

 Let $\FF_n^+$ be the unital free semigroup on $n$ generators
$g_1,\ldots, g_n$ and the identity $g_0$. The length of $\alpha\in
\FF_n^+$ is defined by $|\alpha|:=0$ if $\alpha=g_0$  and
$|\alpha|:=k$ if
 $\alpha=g_{i_1}\cdots g_{i_k}$, where $i_1,\ldots, i_k\in \{1,\ldots, n\}$.  If $Z_1,\ldots, Z_n $ are noncommuting indeterminates   and $\alpha=g_{i_1}\cdots g_{i_k}$,    we
denote $Z_\alpha:= Z_{i_1}\cdots Z_{i_k}$  and $Z_{g_0}:=1$.
Given a positive regular noncommutative polynomial $p:=\sum_{\alpha\in \FF_n^+} a_\alpha
Z_\alpha$, i.e. 
$a_\alpha\geq 0$ for every $\alpha\in \FF_n^+$, \ $a_{g_0}=0$,
 \ and  $a_{g_i}>0$ for all $i\in \{1,\ldots, n\}$, and
   $m\in \NN:=\{1,2,\ldots \}$, we define  the  noncommutative domain
$$
{\cD}_p^m(\cH):=\left\{ X=(X_1,\ldots, X_n)\in B(\cH)^n: \ \Phi_{p,X}(I)\leq I \text{ and } 
( id-\Phi_{p,X})^m(I)\geq 0  \right\},
$$
where $\Phi_{p,X}:B(\cH)\to B(\cH)$ is the completely positive linear map
defined by $\Phi_{p,X}(Y):=\sum_{\alpha\in \FF_n^+} a_\alpha X_\alpha YX_\alpha^*$.
  When $m\geq 1$, $n\geq 1$,   the domain ${\cD}_p^m(\cH)$ was studied  in   \cite{Po-domains} (when $m=1$), and in \cite{Po-Berezin} (when $m\geq 2$).

In a recent paper  \cite{Po-Noncommutative domains}, we provided  large classes of noncommutative   domains  $\cD_{g^{-1}}(\cH)$  in $B(\cH)^n$  associated with {\it admissible free holomorphic functions}  $g= 1+\sum_{ |\alpha|\geq 1} b_\alpha
Z_\alpha$ with $b_\alpha>0$, and studied  these domains and their universal models  ${\bf W}:=(W_1,\ldots, W_n)$ of weighted left creation operators acting on the full Fock space  $F^2(H_n)$ in connection with the Hardy algebras and $C^*$-algebras  they generate. 
The present paper is a continuation of   \cite{Po-Noncommutative domains} aiming at characterizing the invariant subspaces of the universal model   ${\bf W}:=(W_1,\ldots, W_n)$ and at developing a dilation theory for the elements of  the noncommutative  domain
 $\cD_{g^{-1}}(\cH)$. We also obtain results concerning the commutant lifting  and Toeplitz-corrona  in our setting as  well as some results on the  boundary property for universal models.

In Section 1, we  consider some preliminaries  concerning    the class of {\it admissible free holomorphic functions for operator model theory}  which is  in one-to-one  correspondence with a class of universal models ${\bf W}=(W_1,\ldots, W_n)$  of weighted left creation operators acting on the full Fock space  $F^2(H_n)$ and   in one-to-one  correspondence with a class of  maximal noncommutative domains $\cD_{g^{-1}}(\cH)$  in $B(\cH)^n$. We introduce some notation and  recall from \cite{Po-Noncommutative domains}  some 
results regarding these  {\it admissible noncommutative domains}.

 In Section 2, we show that, for any admissible free holomorphic function, the corresponding  universal model ${\bf W}=(W_1,\ldots, W_n)$ admits a Beurling  \cite{Be}  type characterization  of the joint invariant subspaces of $W_1,\ldots, W_n$ (see Theorem \ref{Be-gen}). The uniqueness of the representation of the joint invariant subspaces using partially isometric multi-analytic operators is also addressed. 
 In particular, we prove that if  $g$ is an admissible free holomorphic function such that  $\frac{b_\alpha}{b_{g_i\alpha}}\leq 1$ for any $i\in \{1,\ldots,n\}$ and $\alpha\in \FF_n^+$, 
then    $\cM\subset F^2(H_n)\otimes \cK$ is a joint invariant subspace under  $W_1\otimes I_\cK,\ldots, W_n\otimes I_\cK$ if and only if 
there is a Hilbert space $\cD$ and partial isometry $\Psi:F^2(H_n)\otimes \cD\to F^2(H_n)\otimes \cK$ such that
$$
\cM=\Psi(F^2(H_n)\otimes \cD)
$$
and 
$$
\Psi (S_i\otimes I_\cD)=(W_i\otimes I_\cK) \Psi,\qquad i\in \{1,\ldots, n\},
$$
where $S_1,\ldots, S_n$ are the left creation operators on the full Fock space $F^2(H_n)$. Moreover, the Beurling-Lax-Halmos \cite{Be}, \cite{La}, \cite{Ha} type representation for  the joint invariant subspace  under $W_1\otimes I_\cK,\ldots, W_n\otimes I_\cK$ is essentially unique, that is,  if
$$
\Psi_1\left(F^2(H_n)\otimes \cD_1\right)=\Psi_2\left(F^2(H_n)\otimes \cD_2\right),
$$
where $\Psi_j:F^2(H_n)\otimes \cD_j\to F^2(H_n)\otimes \cK$, $j=1,2$, are partially isometric multi-analytic operators, then there is a partial isometry $V:\cE_1\to \cE_2$ such that $\Psi_1=\Psi_2(I \otimes V)$.  
 
In Section 3, we study the minimal dilations of the pure $n$-tuples of operators in the noncommutative domain $\cD_{g^{-1}}(\cH)$. In addition, we introduce the dilation index and discuss its significance in dilation theory. We show that, under certain natural  conditions, the minimal dilation is unique (see Theorem \ref{unique-minimal}) up to an isomorphism. 

Section 4 is devoted  to dilation theory for not necessarily pure  $n$-tules of operators  in noncommutative domains.   In particular,  if   $\cD_{g^{-1}} $ is  a radially pure domain  such   that $\sum_{k=0}^\infty\sum_{|\alpha|=k} a_\alpha W_\alpha W_\alpha^*$ converges in the operator norm, where  $g^{-1}=\sum_{\alpha\in \FF_n^+} a_\alpha Z_\alpha$,   and  $T=(T_1,\ldots, T_n)\in B(\cH)^n$, then 
$T\in  \cD_{g^{-1}}(\cH)$  if and only if  there is a Hilbert space $\cD$ and  a  Cuntz  \cite{Cu} type $*$-representation  $\pi:C^*({\bf W})\to \cK_\pi$ on a separable Hilbert space $\cK_\pi$    such that 
$$
T_i^*=[(W_i^*\otimes I_\cD)\bigoplus \pi(W_i)^*]_\cH,\qquad  i\in\{1,\ldots, n\}.
$$
  We mention that an operator model theory  in terms of characteristic functions associated with completely non-coisometric $n$-tuples of operators in noncommutative domains $\cD_{g^{-1}}(\cH)$ will be considered  in a future paper. 
  
In Section 5,  we  prove a commutant lifting theorem  (see Theorem \ref{commutant}) and use it to provide factorization results for multi-analytic operators and a Toeplitz-corona theorem (see Theorem \ref{corona2}).

In the last section of the paper,  we provide a large class  of admissible free holomorphic functions with the property that  the corresponding universal model ${\bf W}=(W_1,\ldots, W_n)$ has the boundary property (see  \cite{Arv1-acta}, \cite{Arv2-acta}), i.e.
the identity representation of the  $C^*$-algebra $C^*({\bf W})$ is a boundary representation for the operator space $\Span \{I,W_1,\ldots, W_n\}$.
We  also show that the full Fock space $F^2(H_n)$ is rigid with respect to  the universal model ${\bf W}$.

\bigskip

\section{Universal operator  models associated with  admissible free holomorphic functions}

In this section, we introduce the class of free holomorphic functions admissible for multi-variable 
operator model theory. Each such a function uniquely defines an $n$-tuple ${\bf W}=(W_1,\ldots, W_n)$
of weighted left creation operators acting on the full Fock space with $n$ generators, which turns out to be the universal model for a certain maximal noncommutative domain in $B(\cH)^n$.

 Let $H_n$ be an $n$-dimensional complex  Hilbert space with orthonormal
      basis
      $e_1,\dots,e_n$, where $n\in\NN$.        
      The full Fock space  of $H_n$ is defined by
      $$F^2(H_n):=\bigoplus_{k\geq 0} H_n^{\otimes k},$$
      where $H_n^{\otimes 0}:=\CC 1$ and $H_n^{\otimes k}$ is the (Hilbert)
      tensor product of $k$ copies of $H_n$.
      Define the {\it left creation
      operators} $S_i:F^2(H_n)\to F^2(H_n), \  i\in\{1,\dots, n\}$,  by
      $$
       S_i\varphi:=e_i\otimes\varphi, \quad  \varphi\in F^2(H_n),
      $$
      and  the {\it right creation operators}
      $R_i:F^2(H_n)\to F^2(H_n)$  by
      $
       R_i\varphi:=\varphi\otimes e_i$, \ $ \varphi\in F^2(H_n)$.  
              
Let $\FF_n^+$ be the unital free semigroup on $n$ generators
$g_1,\ldots, g_n$ and the identity $g_0$.  The length of $\alpha\in
\FF_n^+$ is defined by $|\alpha|:=0$ if $\alpha=g_0$  and
$|\alpha|:=k$ if
 $\alpha=g_{i_1}\cdots g_{i_k}$, where $i_1,\ldots, i_k\in \{1,\ldots, n\}$. We set  $e_\alpha:=e_{g_{i_1}}\otimes \cdots \otimes e_{g_{i_k}}$ and  $e_{g_0}=1$, and note that $\{e_\alpha\}_{\alpha\in \FF_n^+}$ is an othonormal basis for $F^2(H_n)$.
 Given a formal power series  $f:= \sum_{\alpha\in \FF_n^+} a_\alpha
Z_\alpha$, \ $a_\alpha\in \CC$, in indeterminates $Z_1,\ldots, Z_n$, we say that $f$ is
a {\it free holomorphic function} in a neighborhood of the origin if there is $\delta>0$ such that  the series
$\sum_{k=0}^\infty \sum_{|\alpha|=k} a_\alpha X_\alpha$ is convergent in norm  for any
$(X_1,\ldots, X_n)\in [B(\cH)^n]_\rho$ and any Hilbert space $\cH$,  where
$$[B(\cH)^n]_\rho:=\{(X_1,\ldots, X_n)\in B(\cH)^n: \
\|X_1X_1^*+\cdots + X_nX_n^*\|^{1/2}\leq \rho\}.
$$
We refer the reader to \cite{Po-holomorphic} for basic properties of free holomorphic functions.   
 Let $\{b_\alpha\}_{\alpha\in \FF_n^+}$ be  a collection of strictly positive  real numbers such that
\begin{equation*}
b_{g_0}=1\quad \text{ and }  \quad
\sup_{\alpha\in \FF_n^+} \frac{b_\alpha}{b_{g_i \alpha}}<\infty,\qquad i\in \{1,\ldots, n\},
\end{equation*}
and let  $g= 1+\sum_{ |\alpha|\geq 1} b_\alpha
Z_\alpha$ be the  associated formal power series.
We  define the {\it weighted left creation  operators}
$W_i:F^2(H_n)\to F^2(H_n)$, $i\in \{1,\ldots, n\}$,  associated with  $g$     by setting $W_i:=S_iD_i$, where
 $S_1,\ldots, S_n$ are the left creation operators on the full
 Fock space $F^2(H_n)$ and the diagonal operators $D_i:F^2(H_n)\to F^2(H_n)$
 are defined by setting
$$
D_ie_\alpha=\sqrt{\frac{b_\alpha}{b_{g_i \alpha}}} e_\alpha,\qquad
 \alpha\in \FF_n^+,
$$
where  $\{e_\alpha\}_{\alpha\in \FF_n^+}$ is the othonormal basis for $F^2(H_n)$. We associate with $g$ the $n$-tuple  ${\bf W}=(W_1,\ldots, W_n)$.

\begin{definition}
A free holomorphic function  $g= 1+\sum_{ |\alpha|\geq 1} b_\alpha
Z_\alpha$  is  said to be  {\it admissible} for  operator model theory if the following conditions are satisfied.
\begin{enumerate}
\item[(a)] The collection  $\{b_\alpha\}_{\alpha\in \FF_n^+}$  consists  of strictly positive  real numbers such that
\begin{equation*}
b_{g_0}=1\quad \text{ and }  \quad
\sup_{\alpha\in \FF_n^+} \frac{b_\alpha}{b_{g_i \alpha}}<\infty,\qquad i\in \{1,\ldots, n\}.
\end{equation*}
\item[(b)] The formal power series $g:= 1+\sum_{ |\alpha|\geq 1} b_\alpha
Z_\alpha$  is  a free holomorphic function in a neighborhood of the origin, which is  equivalent  to  
$
\limsup_{k\to\infty} \left( \sum_{|\alpha|=k}
b_\alpha^2\right)^{1/2k}<\infty. 
$    

\item[(c)]  The condition 
\begin{equation*} 
\sup_{k\in \NN}\left\|\sum_{p=0}^k\sum_{ |\alpha|= p} a_\alpha W_\alpha W_\alpha^*\right\|<\infty
\end{equation*}
holds, where  $g^{-1}=1+\sum_{|\alpha|\geq 1} a_\alpha Z_\alpha $ is the inverse of $g$ and ${\bf W}=(W_1,\ldots, W_n)$ is the $n$-tuple of   the  weighted left creation operators associated with $\{b_\alpha\}_{\alpha\in \FF_n^+}$.
\end{enumerate}

The $n$-tuple ${\bf W}=(W_1,\ldots, W_n)$ of weighted left creation operators is called universal model  associated with $g$.
\end{definition}

\begin{definition}
 Let  $g= 1+\sum_{ |\alpha|\geq 1} b_\alpha
Z_\alpha$ be  a free holomorphic function in a neighborhood of the origin, with $b_\alpha>0$, if $|\alpha|\geq 1$, and with   inverse  $g^{-1}=1+\sum_{ |\gamma|\geq 1} a_\gamma
Z_\gamma$.
The noncommutative   set  $\cD_{g^{-1}}(\cH)$ is defined as the set  of all $n$-tuples  of operators
$X=(X_1,\ldots, X_n) \in B(\cH)^n$ satisfying the following conditions.
\begin{enumerate}
\item[(i)]  $\Delta_{g^{-1}}(X,X^*)\geq 0$, where
$$
\Delta_{g^{-1}}(X,X^*):= \sum_{p=0}^\infty\sum_{ |\alpha|= p} a_\alpha  X_\alpha X_\alpha^*
$$
is convergent in the   weak operator topology. 
\item[(ii)]  
$\sum_{\alpha \in \FF_n^+} b_\alpha X_\alpha \Delta_{g^{-1}}(X,X^*) X_\alpha^*\leq I$, where the convergence is in the weak operator topology.
 
\item[(iii)]  $X$  is said to be a {\it Cuntz tuple} in 
$\cD_{g^{-1}}(\cH)$ if  $\Delta_{g^{-1}}(X,X^*)=0$.
\item[(iv)]  $X$ is said to be a {\it pure}  element in  
$\cD_{g^{-1}}(\cH)$,
if 
$$
\sum_{\alpha \in \FF_n^+} b_\alpha X_\alpha \Delta_{g^{-1}}(X,X^*) X_\alpha^*=I.
$$
 \item[(v)]  
 The {\it pure part} of   $\cD_{g^{-1}}(\cH)$ is defined by setting
$$
\cD^{pure}_{g^{-1}}(\cH):=\{X\in \cD_{g^{-1}}(\cH): X  \text{ is pure}\}.
$$
\end{enumerate}
\end{definition}
We call    $\cD_{g^{-1}}(\cH)$  {\it admissible noncommutative domain} if $g$ is an admissible free holomorphic function for operator model theory. 

\begin{definition}  
We say that $X=(X_1,\ldots, X_n)\in B(\cH)^n$ is  a {\it radial}  $n$-tuple with respect to the set $\cD_{g^{-1}}(\cH)$ if 
  there is  $\delta\in (0,1)$ such  that $rX:=(rX_1,\ldots, rX_n)\in \cD_{g^{-1}}(\cH)$ for every $r\in (\delta, 1)$.  If, in addition, $rX$ is pure for every $r\in (\delta, 1)$, we say that $X$ is  a radially pure $n$-tuple with respect to  $\cD_{g^{-1}}(\cH)$. The noncommutative set $ \cD_{g^{-1}}(\cH)$ is called  {\it radially pure} domain if   any $X\in \cD_{g^{-1}}(\cH)$ is  a radially pure $n$-tuple.
  \end{definition}

\begin{definition}  Let $T=(T_1,\ldots, T_n)\in B(\cH)^n$.
If there is    a Hilbert space $\cE$ and an isometric operator $V:\cH\to F^2(H_n)\otimes \cE$ such that 
$$
VT_i^*=(W_i^*\otimes I_\cE)V,\qquad i\in \{1,\ldots,n\},
$$ 
we say that ${\bf W}=(W_1,\ldots, W_n)$ is a {\it universal model} for $T$ . 
\end{definition}

We define the
{\it noncommutative Berezin kernel} associated with an arbitrary  $n$-tuple $T=(T_1,\ldots,
T_n)$  in $\cD_{g^{-1}}(\cH)$ to be  the operator $K_{g,T}:\cH\to
F^2(H_n)\otimes \overline{\Delta_{g^{-1}} (T,T^*) (\cH)}$  given by
\begin{equation*}
 K_{g,T}h:=\sum_{\alpha\in \FF_n^+} \sqrt{b_\alpha}
e_\alpha\otimes \Delta_{g^{-1}}(T,T^*)^{1/2} T_\alpha^* h,\quad h\in \cH.
\end{equation*}

We proved in   \cite{Po-Noncommutative domains} the following result.

\begin{theorem} \label{model}  Let $g$ be an admissible  free holomorphic function and let $T=(T_1,\ldots, T_n)\in B(\cH)^n$. Then  there is a Hilbert spaces $\cD$ such that 
$$
T_i^*=(W_i^*\otimes I_\cD)|_{\cH},\qquad i\in \{1,\ldots, n\},
$$
where $\cH$ is  identified with a coinvariant subspace  for the operators $W_1\otimes I_\cD,\ldots, W_n\otimes I_\cD$, if and only  if $T$  is    a  pure $n$-tuple in the noncommutative set  $\cD_{g^{-1}}(\cH)$.
In this case, we have
 $$
 K_{g,T} T_i^*=(W_i^*\otimes I_\cD)K_{g,T},\qquad i\in \{1,\ldots, n\},
 $$
 and the noncommutative Berezin kernel $K_{g,T}$ is an isometry.
\end{theorem}
 
If  ${\bf W}=(W_1,\ldots, W_n)$ is the universal model  associated with $g$, we  introduce the {\it noncommutative domain algebra} $\cA(g)$ as the norm-closed non-self-adjoint algebra generated by $W_1,\ldots, W_n$ and the the identity.  
 
  Now, we recall from \cite{Po-Noncommutative domains}
  several classes of admissible free holomorphic functions  for  operator model theory and the associated noncommutative domains. The examples presented  here will be referred to throughout the paper.

\begin{definition}  We say that  the noncommutative  set $\cD_{g^{-1}}$ is a regular domain if $g$ is a  free holomorphic function in a neighborhood of the origin   such that
$$
g^{-1} =1+\sum_{\alpha\in \FF_n^+, |\alpha|\geq 1} a_\alpha Z_\alpha
$$
for some coefficients $a_\alpha\leq  0$ and $a_{g_i}<0$ for every $i\in \{1,\ldots, n\}$. 
\end{definition}

\begin{example}  \label{Bergman}
 For each $s\in (0,\infty)$, consider  the formal power series
$$
g_s:=1+\sum_{k=1}^\infty\left(\begin{matrix} s+k-1 \\k \end{matrix}\right)(Z_1+\cdots  +Z_n)^k.
$$ 
  If $s\in (0,1]$, 
  $\cD_{g_s^{-1}}$ is a regular domain.
When  $s\in \NN$, the corresponding noncommutative domain $\cD_{g_s^{-1}}$ was extensively studied in  \cite{Po-Berezin}, \cite{Po-domains}, and \cite{Po-Bergman}. Note that
$$
g_s^{-1}=\left[1-(Z_1+\cdots +Z_n)\right]^s=\sum_{k=0}^\infty (-1)^k\left(\begin{matrix} s \\k \end{matrix}\right)(Z_1+\cdots + Z_n)^k.
$$
Therefore 
$$a_\alpha=(-1)^{|\alpha|} \left(\begin{matrix} s \\|\alpha| \end{matrix}\right)=(-1)^{|\alpha|}\frac{s(s-1)\cdots (s-|\alpha|+1)}{|\alpha| !},\qquad \alpha\in \FF_n^+, |\alpha|\geq 1.
$$
It is easy to see  that there is $N_s\geq 2$ such that $a_\alpha\geq 0$ $($or $a_\alpha\leq 0)$ for every $\alpha\in \FF_n^+$ with $|\alpha|>N_s$.
According to  Theorem 4.7 from \cite{Po-Noncommutative domains},   $g_s$ is an admissible free holomorphic function. In addition, if  $s\in [1,\infty)$, then 
 $
\sum_{k=0}^\infty \left|\left(\begin{matrix} s \\k \end{matrix}\right)\right|<\infty$   and  
    the series $\sum_{\alpha\in \FF_n^+} a_\alpha W_\alpha W_\alpha^*$ converges in the operator norm topology.

 \end{example}

\begin{example} \label{Dirichlet}
For each $s\in \RR$  the free holomorphic function
$$
\xi_s:=1+\sum_{\alpha\in \FF_n^+, |\alpha|\geq 1} (|\alpha|+1)^s Z_\alpha
$$
is admissible.
If $s\leq 0$,   $\cD_{\xi_s^{-1}}$ is a regular domain.  
\end{example}
 For more examples of admissible free holomorphic functions we refer the reader to our recent paper 
 \cite{Po-Noncommutative domains}.
  \bigskip

\section{Invariant subspaces   for  universal  operator models}

In this section, we show that, for any admissible free holomorphic function, the corresponding  universal model ${\bf W}=(W_1,\ldots, W_n)$ admits a Beurling type characterization  of the joint invariant subspaces of $W_1,\ldots, W_n$. The uniqueness of the representation of the joint invariant subspaces using partially isometric multi-analytic operators is also addressed. All these results apply, in particular, to the universal models associated with     the free holomorphic  functions $g_s$, $ s\in(0,\infty)$ (see Example \ref{Bergman}).

First, we recall from  \cite{Po-Noncommutative domains} the following result on the universal model  ${\bf W}=(W_1,\ldots, W_n)$  associated with  an admissible free holomorphic function.

\begin{proposition}   \label{W}  Let $g= 1+\sum_{ |\alpha|\geq 1} b_\alpha
Z_\alpha$ be an admissible free holomorphic function with inverse $g^{-1}=1+\sum_{ |\gamma|\geq 1} a_\gamma
Z_\gamma$ and let ${\bf W}=(W_1,\ldots, W_n)$ be  the corresponding  weighted left creation operators associated with $g$.
Then 
$$
\Delta_{g^{-1}}({\bf W},{\bf W}^*):=\sum_{p=0}^\infty\sum_{ |\alpha|= p} a_\alpha  W_\alpha W_\alpha^*
$$
converges in the  strong operator topology and  $\Delta_{g^{-1}}({\bf W},{\bf W}^*)=P_{\CC 1}$, the orthogonal projection of $F^2(H_n)$ onto the constants $\CC 1$. Moreover, we have
$$
\sum_{\alpha \in \FF_n^+} b_\alpha W_\alpha \Delta_{g^{-1}}({\bf W},{\bf W}^*) W_\alpha^*=I,
$$
where the convergence is in the strong operator topology.
\end{proposition}

Let  ${\bf W}^{(g)}:=(W_1^{(g)},\ldots, W_n^{(g)})$ be the universal model of the noncommutative domain 
$\cD_{g^{-1}}$. We denote by $\cA(g)$ the closed non-selfadjoint  algebra generated by $W_1^{(g)},\ldots, W_n^{(g)} $ and the identity.  The $C^*$-algebra  generated by $W_1^{(g)},\ldots, W_n^{(g)} $ and the identity will be denoted by $C^*({\bf W}^{(g)})$. Given two domains  $\cD_{g^{-1}}$ and $\cD_{f^{-1}}$, if   
$\cD^{pure}_{f^{-1}}(\cH)\subset \cD^{pure}_{g^{-1}}(\cH)$ for any Hilbert space 
$\cH$, we write $\cD^{pure}_{f^{-1}}\subset \cD^{pure}_{g^{-1}}$.

\begin{theorem}  \label{inclusion} Let $f$ and $g$ be two admissible free holomorphic functions and let 
$${\bf W}^{(f)}:=(W_1^{(f)},\ldots, W_n^{(f)})\quad  \text{ and } \quad  {\bf W}^{(g)}:=(W_1^{(g)},\ldots, W_n^{(g)})
$$
 be the universal models associated with the domains $\cD_{f^{-1}}$ and $\cD_{g^{-1}}$, respectively. Then the following statements are equivalent.
\begin{enumerate}
\item[(i)] $\cD^{pure}_{f^{-1}}\subset \cD^{pure}_{g^{-1}}$.

\item[(ii)] ${\bf W}^{(f)}$ is a pure $n$-tuple in $ \cD_{g^{-1}}(F^2(H_n))$.

\item[(iii)] There is a Hilbert space $\cG$ and an isometry $V:F^2(H_n)\to F^2(H_n)\otimes \cG$ such that 
$$
VW_i^{(f)*}=(W_i^{(g) *}\otimes I_\cG)V,\qquad i\in \{1,\ldots, n\}.
$$
\end{enumerate}
If, in addition,  $ \cD_{g^{-1}}$ is a regular domain, then the conditions above are equivalent to the following.
\begin{enumerate}
\item[(iv)]  The linear map $\Psi: \cA(g)\to \cA(f)$ defined by 
$
\Psi(W_\alpha^{(g)}):=W_\alpha^{(f)}$,  $\alpha\in \FF_n^+,
$ 
is completely contractive and has  a coextension to a $*$-representation  $\pi:C^*({\bf W}^{(g)})\to B(\cK)$ which is pure, i.e.
$$
\Psi(W_\alpha^{(g)})^*=\pi(W_\alpha^{(g)})^*|_{F^2(H_n)}, \qquad \alpha\in \FF_n^+,
$$
and $\pi$ is unitarily equivalent to a multiple of the identity representation of $C^*({\bf W}^{(g)})$.
\end{enumerate}
\end{theorem}
\begin{proof}  According to Proposition \ref{W},   ${\bf W}^{(f)}$ is a pure $n$-tuple in $ \cD_{f^{-1}}(F^2(H_n))$. Consequently, the implication (i)$\implies$(ii) is clear.
To prove that (ii)$\implies$(iii), suppose  that   ${\bf W}^{(f)}$ is a pure $n$-tuple in 
$ \cD_{g^{-1}}(F^2(H_n))$. Due to Theorem \ref{model},
 we have 
$$
 K_{g,{\bf W}^{(f)}} {W_i^{(f)*}}=({W_i^{(g)*}}\otimes I_\cD)K_{g,{\bf W}^{(f)}},\qquad i\in \{1,\ldots, n\}
 $$
 and the noncommutative Berezin kernel $K_{g,{\bf W}^{(f)}}$ is an isometry. Taking $V:=K_{g,{\bf W}^{(f)}}$, we conclude that item (iii) holds.
 Now, we prove that (iii)$\implies$(i).  To this end, suppose that (iii) holds. Assume that $g=1+\sum_{|\alpha|\geq 1} b_\alpha Z_\alpha$
 and $g^{-1}=1+\sum_{|\alpha|\geq 1} a_\alpha Z_\alpha$.  Then, due to Proposition \ref{W},
 $$
\Delta_{g^{-1}}({\bf W}^{(g)},{\bf W}^{(g)*}):=\sum_{p=0}^\infty\sum_{ |\alpha|= p} a_\alpha  W_\alpha^{(g)} W_\alpha^{(g)*}
$$
converges in the  strong operator topology and  $\Delta_{g^{-1}}({\bf W}^{(g)},{\bf W}^{(g)*})=P_{\CC1}$, the orthogonal projection of $F^2(H_n)$ onto the constants $\CC1$.
Consequently,
$$
\Delta_{g^{-1}}({\bf W}^{(f)},{\bf W}^{(f)*}):=\sum_{p=0}^\infty\sum_{ |\alpha|= p} a_\alpha  W_\alpha^{(f)} W_\alpha^{(f)*}=V^*\left(\sum_{p=0}^\infty\sum_{ |\alpha|= p} a_\alpha  W_\alpha^{(g)} W_\alpha^{(g)*}\right)
V\geq 0,
$$
where  the series converge in the  strong operator topology. On the other hand, using again Proposition \ref{W}, we have
$$
\sum_{\alpha \in \FF_n^+} b_\alpha W^{(f)}_\alpha\Delta_{g^{-1}}({\bf W}^{(f)},{\bf W}^{(f)*}) W^{(f)*}_\alpha
=
V^*\left(\sum_{\alpha \in \FF_n^+} b_\alpha W_\alpha^{(g)}\Delta_{g^{-1}}({\bf W}^{(g)},{\bf W}^{(g)*}) W_\alpha^{(g)*}\right)V= V^*V=I,
$$
where the convergence is in the strong operator topology. Hence, ${\bf W}^{(f)}\in \cD_{g^{-1}}(F^2(H_n))$ is a pure $n$-tuple.
 
 Now,  assume that   $\cD_{g^{-1}}$ is a regular domain.  To  prove the implication  (iii)$\implies$(iv), suppose that (iii) holds.
 Since ${\bf W}^{(f)}$ is a pure $n$-tuple in $ \cD_{g^{-1}}(F^2(H_n))$, 
 Corollary 1.14 from  \cite{Po-Noncommutative domains}  shows that  the linear map $\Psi: \cA(g)\to \cA(f)$ defined by 
$
\Psi(W_\alpha^{(g)}):=W_\alpha^{(f)}$, $ \alpha\in \FF_n^+,
$ 
is  a completely contractive homomorphism. 
 Consider the $*$-representation  $\pi: C^*({\bf W}^{(g)})\to B(F^2(H_n)\otimes \cG)$ defined by $\pi(A):=A\otimes I_\cG$ for  any $A\in C^*({\bf W}^{(g)})$.
Since  $
VW_i^{(f)*}=(W_i^{(g) *}\otimes I_\cG)V$ for $i\in \{1,\ldots, n\}$  and 
$V:F^2(H_n)\to F^2(H_n)\otimes \cG$  is an isometry, we can identify $F^2(H_n)$ with $V(F^2(H_n))$. Consequently,  $F^2(H_n)$ is invariant under each  operator $W_i^{(g) *}\otimes I_\cG$ and 
$
W_\alpha^{(f)*}= \pi(W_\alpha^{(g)})^*|_{F^2(H_n)},  \alpha\in \FF_n^+.
$
To complete the proof, we show  that (iv)$\implies$(ii). To this end, suppose that item (iv) holds.
Therefore, there is a $*$-representation  $\pi:C^*({\bf W}^{(g)})\to B(\cK)$ which is pure, i.e.
$
W_\alpha^{(f)*}=\pi(W_\alpha^{(g)})^*|_{F^2(H_n)}, \alpha\in \FF_n^+,
$
and $\pi$ is unitarily equivalent to a multiple of the identity representation of $C^*({\bf W}^{(g)})$. 
 Since $ \cD_{g^{-1}}$ is a regular domain, we have  and $g^{-1}=1-\sum_{|\alpha|\geq 1} c_\alpha Z_\alpha$, where $c_\alpha\geq 0$  and $c_{g_i}>0$ for $i\in \{1,\ldots, n\}$.
 Note that
 \begin{equation*}
 \begin{split}
 \left\| \sum_{1\leq|\alpha|\leq m} c_\alpha W_\alpha^{(f)} W_\alpha^{(f)*}\right\|
 &\leq \left\| P_{F^2(H_n)}\pi\left( \sum_{1\leq|\alpha|\leq m} c_\alpha W_\alpha^{(g)} W_\alpha^{(g)*}\right)|_{F^2(H_n)}\right\| \\
 &\leq \left\| \sum_{1\leq|\alpha|\leq m} c_\alpha W_\alpha^{(g)} W_\alpha^{(g)*}\right\|\leq 1
 \end{split}
 \end{equation*}
 for any $m\in \NN$.
 Consequently,  $ \sum_{|\alpha|\geq 1} c_\alpha W_\alpha^{(f)} W_\alpha^{(f)*}\leq I$, which shows that
 ${\bf W}^{(f)}\in \cD_{g^{-1}}(F^2(H_n))$.
 On the other hand, we know that $\pi(W_\alpha^{(g)})=U^*(W_\alpha^{(g)}\otimes I_\cG)U$, where 
 $U:\cK\to  F^2(H_n)\otimes  \cG$ is a  unitary operator, and 
 $$
 \sum_{\alpha \in \FF_n^+} b_\alpha W^{(g)}_\alpha\Delta_{g^{-1}}({\bf W}^{(g)},{\bf W}^{(g)*}) W^{(g)*}_\alpha=I_{F^2(H_n)}
 $$
where the convergence is in the strong topology.  Hence, we deduce that
$$
 \sum_{\alpha \in \FF_n^+} b_\alpha \pi(W^{(g)}_\alpha)\Delta_{g^{-1}}(\pi({\bf W}^{(g)}),\pi({\bf W}^{(g)*}) )\pi(W^{(g)*}_\alpha)=I_\cK
 $$
 Taking the compression  to  $F^2(H_n)$ and  using the fact that
   $W_\alpha^{(f)*}=\pi(W_\alpha^{(g)})^*|_{F^2(H_n)}$ for $\alpha\in \FF_n^+,$ we obtain
 $$
 \sum_{\alpha \in \FF_n^+} b_\alpha W^{(f)}_\alpha\Delta_{g^{-1}}({\bf W}^{(f)},{\bf W}^{(f)*}) W^{(f)*}_\alpha=I_{F^2(H_n)},
 $$  
 which shows that  ${\bf W}^{(f)}$ is a pure $n$-tuple in $\cD_{g^{-1}}(F^2(H_n))$. Therefore item (ii) holds.
 The proof is complete.
\end{proof}

The following theorem  extends the corresponding result  from \cite{Po-invariant} to our more general setting.   

\begin{theorem} \label{Beurling}  Let $f$ and $g$ be two admissible free holomorphic functions and let 
$${\bf W}^{(f)}:=(W_1^{(f)},\ldots, W_n^{(f)})\quad  \text{ and } \quad  {\bf W}^{(g)}:=(W_1^{(g)},\ldots, W_n^{(g)})
$$
 be the universal models associated with $f$ and $g$, respectively. If $\cD_{g^{-1}}$ is a regular domain such that
 $\cD^{pure}_{f^{-1}}\subset \cD^{pure}_{g^{-1}}$
and $Y\in B(F^2(H_n)\otimes \cK)$ is a  self-adjoint operator,    then the following statements are equivalent.
\begin{enumerate}
\item[(i)] There is a Hilbert space $\cD$ and an operator  $\Psi:F^2(H_n)\otimes \cD\to  F^2(H_n)\otimes \cK$ such that
 $Y=\Psi\Psi^*$ and 
 $$\Psi(W_i^{(g)}\otimes I_\cD)=(W_i^{(f)}\otimes I_\cK)\Psi, \qquad i\in \{1,\ldots,n\}.
 $$

\item[(ii)] $Y- \sum_{|\alpha|\geq 1} c_\alpha (W_\alpha^{(f)}\otimes I_\cK) Y (W_\alpha^{(f)*}\otimes I_\cK)\geq 0$, where $g^{-1}=1-\sum_{|\alpha|\geq 1} c_\alpha Z_\alpha$.

\end{enumerate}

\end{theorem}
\begin{proof} Suppose that  item (i) holds.  Since $(W_1^{(g)},\ldots, W_n^{(g)})\in \cD_{g^{-1}}(F^2(H_n))$ and $\cD_{g^{-1}}$ is a regular domain, we have $c_\alpha\geq 0$ and 
\begin{equation*}
\begin{split}
\sum_{|\alpha|\geq 1} c_\alpha (W_\alpha^{(f)}\otimes I_\cK) Y (W_\alpha^{(f)*}\otimes I_\cK)
&=\sum_{|\alpha|\geq 1} c_\alpha (W_\alpha^{(f)}\otimes I_\cK) \Psi\Psi^* (W_\alpha^{(f)*}\otimes I_\cK)\\
&=\Psi\left( \sum_{|\alpha|\geq 1} c_\alpha (W_\alpha^{(g)} W_\alpha^{(g)*}\otimes I_\cD)\right)
 \Psi^*\leq  \Psi\Psi^*=Y,
 \end{split}
\end{equation*}
which proves the implication (i)$\implies$(ii). Conversely, assume that inequality (ii) holds.
  Define the completely  positive linear map $\phi_{{\bf W}^{(f)}}:B(F^2(H_n)\otimes \cK)\to B(F^2(H_n)\otimes \cK)$ by setting 
$\phi_{{\bf W}^{(f)}}(T):=\sum_{|\alpha|\geq 1} c_\alpha (W_\alpha^{(f)}\otimes I_\cK) T (W_\alpha^{(f)*}\otimes I_\cK)$.  According to Proposition   \ref{W},  
 ${\bf W}^{(f)}:=(W_1^{(f)},\ldots, W_n^{(f)})$  is a pure $n$-tuple in $ \cD_{f^{-1}}(F^2(H_n))$.  Consequently, since $\cD^{pure}_{f^{-1}}\subset \cD^{pure}_{g^{-1}}$, we  have
SOT-$\lim_{k\to \infty} \phi_{{\bf W}^{(f)}}^k(I)=0$. Taking into account condition (ii), we have  $\phi_{{\bf W}^{(f)}}(Y)\leq Y$, which implies
$$
-\|Y\|\phi_{{\bf W}^{(f)}}^k(I)\leq \phi_{{\bf W}^{(f)}}^k(Y)\leq \|Y\| \phi_{{\bf W}^{(f)}}^k(I).
$$
Hence, we deduce that SOT-$\lim_{k\to \infty} \phi_{{\bf W}^{(f)}}^k(Y)=0$ which together with inequality   $\phi_{{\bf W}^{(f)}}^k(Y)\leq Y$ implies $Y\geq 0$.
For each $i\in \{1,\ldots, n\}$,  define the operator  $E_i:Y^{1/2}(F^2(H_n)\otimes \cK)\to Y^{1/2}(F^2(H_n)\otimes \cK)$ by setting 
\begin{equation} \label{Ei}
E_i(Y^{1/2}y):=Y^{1/2}(W_i^{(f)*}\otimes I_\cK) y,\qquad y\in F^2(H_n)\otimes \cK.
\end{equation}
Since  $\cD_{g^{-1}}$ is a regular domain, we have $c_{g_i}>0$ and, consequently,  
\begin{equation*}
\begin{split}
c_{g_i}\|E_iY^{1/2}y\|^2&=\|\sqrt{c_{g_i}} Y^{1/2} (W_i^{(f)*}\otimes I_\cK)y\|^2=c_{g_i}
 \left<(W_i^{(f)}\otimes I_\cK)Y(W_i^{(f)*}\otimes I_\cK)y,y\right>\\
&\leq  \left<\phi_{{\bf W}^{(f)}}(Y)y,y\right>\leq \left<Yy,y\right>=\|Y^{1/2}y\|^2
\end{split}
\end{equation*}
for any $y\in F^2(H_n)\otimes \cK$. One can see now that the operator $E_i$ is well-defined and has a unique extension,
also denoted by $E_i$, to a bounded operator on $\cE:=\overline{Y^{1/2}(F^2(H_n)\otimes \cK)}$.
Set $F_i:=E_i^*\in B(\cE)$ and note that
\begin{equation*}
\begin{split}
\sum_{1\leq |\alpha|\leq m} c_\alpha \|F_\alpha^* Y^{1/2}y\|^2
&= \sum_{1\leq |\alpha|\leq m}\left\| \sqrt{c_\alpha} Y^{1/2} {W}_\alpha^{(f)*}y\right\|^2\\
&\leq  \left<\phi_X(Y)y,y\right>\leq \|Y^{1/2}y\|^2
\end{split}
\end{equation*}
for any $y\in \cK$ and $m\geq 1$. Consequently, $\sum_{|\alpha|\geq 1}c_\alpha F_\alpha F_\alpha^*\leq I$, which shows that
$F=(F_1,\ldots, F_n)\in \cD_{g^{-1}}(\cE)$. Since SOT-$\lim_{k\to \infty} \phi_{{\bf W}^{(f)}}^k(I)=0$ and 
$$
\left<\phi_F^k(I)Y^{1/2}y, Y^{1/2} y\right>=\left<\phi_{{\bf W}^{(f)}}^k(Y)y, y\right>
\leq \|Y\|\left<\phi_{{\bf W}^{(f)}}^k(I)y, y\right>
$$
for any $k\in \NN$ and $y\in F^2(H_n)\otimes \cK$ and  $\|\phi_F^k(I)\|\leq 1$ for any $k\in \NN$, we conclude that 
SOT-$\lim_{k\to \infty} \phi_F^k(I)=0$, which shows that  $F\in \cD^{pure}_{g^{-1}}(\cE)$.
According to Theorem \ref{model}, we have
\begin{equation}
\label{KF}
 K_{g,F} F_i^*=(W^{(g)*}_i\otimes I_\cD)K_{g,F},\qquad i\in \{1,\ldots, n\},
 \end{equation}
 where $\cD:=\overline{\text{\rm range}\,\Delta_{g^{-1}}(F,F^*)}$, and the noncommutative Berezin kernel $K_{g,F}$ is an isometry. Now, we define the operator  $\Psi:F^2(H_n)\otimes \cD\to F^2(H_n)\otimes \cK$ by setting
 $\Psi:=Y^{1/2} K_{g,F}^*$. Due to relation \eqref{Ei}, we deduce that $Y^{1/2} F_i=(W_i^{(f)}\otimes I_\cK) Y^{1/2}$. Combining the later relation with  \eqref {KF}, we obtain
 \begin{equation*}
 \begin{split}
 \Psi(W_i^{(g)}\otimes I_\cD)&=Y^{1/2} K_{g,F}^*(W_i^{(g)}\otimes I_\cD) =Y^{1/2} F_i K_{g,F}^*\\
 &=(W_i^{(f)}\otimes I_\cK)Y^{1/2} K_{g,F}^*=(W_i^{(f)}\otimes I_\cK)\Psi
 \end{split}
 \end{equation*}
for any $i\in\{1,\ldots, n\}$. Since $K_{g,F}$ is an isometry, we deduce that
$\Psi\Psi^*=Y^{1/2} K_{g,F}^*  K_{g,F}Y^{1/2}=Y$. This completes the proof.
\end{proof}

\begin{theorem} \label{inv-subspace}
Let $f$ and $g$ be two admissible free holomorphic functions and let 
$${\bf W}^{(f)}:=(W_1^{(f)},\ldots, W_n^{(f)})\quad  \text{ and } \quad  {\bf W}^{(g)}:=(W_1^{(g)},\ldots, W_n^{(g)})
$$
 be the universal models associated with $f$ and $g$, respectively. If $\cD_{g^{-1}}$ is a regular domain and 
 $$\cD^{pure}_{f^{-1}}\subset \cD^{pure}_{g^{-1}},$$
    then the following statements are equivalent.
\begin{enumerate}
\item[(i)] $\cM\subset F^2(H_n)\otimes \cK$ is a joint invariant subspace under  $W_1^{(f)}\otimes I_\cK,\ldots, W_n^{(f)}\otimes I_\cK$.
\item[(ii)] There is a Hilbert space $\cD$ and a partial isometry $\Psi:F^2(H_n)\otimes \cD\to F^2(H_n)\otimes \cK$ such that
$$
\cM=\Psi(F^2(H_n)\otimes \cD)
$$
and 
$$
\Psi (W_i^{(g)}\otimes I_\cD)=(W_i^{(f)}\otimes I_\cK) \Psi,\qquad i\in \{1,\ldots, n\}.
$$
\end{enumerate}
\end{theorem}

\begin{proof}

Suppose that (i) holds.
Then $P_\cM (W_i^{(f)}\otimes I_\cK) P_\cM=(W_i^{(f)}\otimes I_\cK) P_\cM$ for every $i\in \{1,\ldots, n\}$. According to Proposition \ref{W},   ${\bf W}^{(f)}:=(W_1^{(f)},\ldots, W_n^{(f)})$  is a pure $n$-tuple in $ \cD_{f^{-1}}(F^2(H_n))$. 
Since
 $ \cD^{pure}_{f^{-1}}\subset \cD^{pure}_{g^{-1}}$,  we also  have   ${\bf W}^{(f)}\in  \cD^{pure}_{g^{-1}}(F^2(H_n))$. Since $\cD_{g^{-1}}$ is a regular noncommutative domain, we have  $g^{-1}=1-\sum_{|\alpha|\geq 1} c_\alpha Z_\alpha$ for some positive coefficients $\{c_\alpha\}_{\alpha\in \FF_n^+}$ with $c_{g_i}>0$. Therefore,  we have   $\sum_{|\alpha|\geq 1} c_\alpha W_\alpha^{(f)} W_\alpha^{(f)*}\leq I$ which implies 
 $$
\sum_{|\alpha|\geq 1} c_\alpha (W_\alpha^{(f)}\otimes I_\cK) P_\cM (W_\alpha^{(f)*}\otimes I_\cK)=P_\cM\left(\sum_{|\alpha|\geq 1} c_\alpha (W_\alpha^{(f)}\otimes I_\cK) P_\cM (W_\alpha^{(f)*}\otimes I_\cK)\right)P_\cM\leq P_\cM.
$$
Applying Theorem \ref{Beurling},  when $Y=P_\cM$, we deduce item (ii). The implication (ii)$\implies$(i) is obvious.
 The proof is complete.
\end{proof}

\bigskip

\begin{corollary}  \label{inv}   Let  $f$ be  an admissible free holomorphic function such that  $\frac{b_\alpha}{b_{g_i\alpha}}\leq 1$ for any $i\in \{1,\ldots,n\}$ and $\alpha\in \FF_n^+$.
Then   the following statements are equivalent.
\begin{enumerate}
\item[(i)] $\cM\subset F^2(H_n)\otimes \cK$ is a joint invariant subspace under  $W_1^{(f)}\otimes I_\cK,\ldots, W_n^{(f)}\otimes I_\cK$.
\item[(ii)] There is a Hilbert space $\cD$ and partial isometry $\Psi:F^2(H_n)\otimes \cD\to F^2(H_n)\otimes \cK$ such that
$$
\cM=\Psi(F^2(H_n)\otimes \cD)
$$
and 
$$
\Psi (S_i\otimes I_\cD)=(W_i^{(f)}\otimes I_\cK) \Psi,\qquad i\in \{1,\ldots, n\},
$$
\end{enumerate}
where $S_1,\ldots, S_n$ are the left creation operators on the full Fock space $F^2(H_n)$.
\end{corollary}
\begin{proof}  
Since   $f$ is  an admissible free holomorphic function such that  $\frac{b_\alpha}{b_{g_i\alpha}}\leq 1$ for any $i\in \{1,\ldots,n\}$ and $\alpha\in \FF_n^+$,  the associated model $(W_1^{(f)},\ldots, W_n^{(f)})$ is a row contraction, i.e. 
$$W_1^{(f)}W_1^{(f)*}+\cdots +W_n^{(f)}W_n^{(f)*}\leq I.$$
Indeed, define  the completely positive linear map $\varphi_{{\bf W}^{(f)}}:B(F^2(H_n))\to B(F^2(H_n))$ by setting
$\varphi_{{\bf W}^{(f)}}(X):=\sum_{i=1}^n W_i^{(f)} XW_i^{(f)*}$, where $X\in B(F^2(H_n))$.
Since $W_i^{(f)}=S_i D_i^{(f)}$ with $\|D_i^{(f)}\|\leq 1$, we deduce that
$$
\varphi_{{\bf W}^{(f)}}(I)=\sum_{i=1}^n S_i[D_i^{(f)}]^2 S_i^*\leq \varphi_S(I)\leq I.
$$
Consequently,  $\|\varphi^k_{{\bf W}^{(f)}}(I)\|\leq 1$ for any $k\in \NN$. Due to the definition of the weighted left creation operators $W_1,\ldots, W_n$, for every  $\alpha\in \FF_n^+$,  
$\varphi^k_{{\bf W}^{(f)}}(I)e_\alpha \to 0$ as $k\to \infty$. Now, it is clear that  $\varphi^k_{{\bf W}^{(f)}}(I)\to 0$ strongly as $k\to\infty$.
Therefore,  ${\bf W}^{(f)}$ is a pure $n$-tuple in $  \cD_{g^{-1}}(F^2(H_n))$, where $g^{-1}:=1-(Z_1+\cdots +Z_n)$. In this case, we have
$\cD_{g^{-1}}(F^2(H_n))=[B(F^2(H_n))^n]_1$ and the associated universal model is  the $n$-tuple 
$(S_1,\ldots, S_n)$ of left creation operators.  Applying Theorem \ref{inv-subspace}, we can complete the proof.
\end{proof}

\begin{proposition} \label{inclusions}  Let $g= 1+\sum_{ |\alpha|\geq 1} b_\alpha
Z_\alpha$ be an admissible free holomorphic function with inverse 
$g^{-1}=1+\sum_{ |\alpha|\geq 1} a_\alpha Z_\alpha.$ 
Define 
$$
g_1:= 1+\sum_{k=1}^\infty \left( \sum_{|\alpha|\geq 1} \frac{|a_\alpha|}{c} Z_\alpha\right)^k,\quad \text{ where } \  0<c<1,
$$
 and 
$$
g_2:= 1+\sum_{k=1}^\infty \left(\frac{Z_1+\cdots+ Z_n}{d}\right)^k,\quad \text{ where }  \ d=\sup_{{\alpha\in \FF_n^+}\atop{i= 1,\ldots, n}} \frac{b_\alpha}{b_{g_i\alpha}}. 
$$
 Then $\cD_{g_1^{-1}}$  and  $\cD_{g_2^{-1}}$ are regular domains and  $\cD^{pure}_{g_1^{-1}}\subset \cD^{pure}_{g^{-1}}\subset \cD^{pure}_{g_2^{-1}}$.
\end{proposition}
\begin{proof} The fact that $\cD_{g_1^{-1}}$  and  $\cD_{g_2^{-1}}$ are regular domains is easy to see.
If $X=(X_1,\ldots, X_n)\in \cD_{g_1^{-1}}(\cH)$, then  $I- \sum_{|\alpha|\geq 1} \frac{|a_\alpha|}{c} X_\alpha X_\alpha^* \geq 0$. Using  Proposition  1.11 from  \cite{Po-Noncommutative domains} , we deduce that  $X\in \cD^{pure}_{g^{-1}}(\cH)$. consequently,   $\cD^{pure}_{g_1^{-1}}\subset \cD^{pure}_{g^{-1}}$.
Let ${\bf W}^{(g)}=(W_1^{(g)},\ldots, W_n^{(g)})$ be the universal model of the domain $\cD_{g^{-1}}$. According to the definition of the operator $W_i^{(g)}$, we have   $\|W_i^{(g)}\|\leq \sqrt{d}$. Since  the operators $W_1^{(g)},\ldots, W_n^{(g)}$  have orthogonal ranges we deduce that
$ \sum_{i=1}^n \frac{1}{d}W_i^{(g)}W_i^{(g)*}\leq I$. This shows that  ${\bf W}^{(g)}\in \cD_{g_2^{-1}}$, where 
$g_2^{-1}:=1-\sum_{i=1}^n \frac{1}{d} Z_i$.  According to Theorem \ref{inclusion},
to show that $\cD^{pure}_{g^{-1}}\subset \cD^{pure}_{g_2^{-1}}$, it is enough to prove that ${\bf W}^{(g)}\in \cD^{pure}_{g_2^{-1}}(F^2(H_n))$. To this end,  let us   prove that  SOT-$\lim_{k\to \infty}\Gamma_{{\bf W}^{(g)}}^k(I)=0$, where 
$\Gamma_{{\bf W}^{(g)}}(Y):=\sum_{i=1}^n \frac{1}{d}W_i^{(g)}YW_i^{(g)*}$. Indeed, since
$W_i^{(g)}=S_i D_i^{(g)}$, where $S_1,\ldots,  S_n$ are the left creation operators, we deduce that $\|D_i^{(g)}\|\leq \sqrt{d}$. Consequently,  $\Gamma_{{\bf W}^{(g)}}(I)\leq \sum_{i=1}^n S_iS_i^*\leq I$, which implies 
$\|\Gamma^k_{{\bf W}^{(g)}}(I)\|\leq 1$ for any $k\in \NN$. On the other hand, it is easy to see that,  for every  $\alpha\in \FF_n^+$,  
$\Gamma^k_{{\bf W}^{(g)}}(I)e_\alpha \to 0$ as $k\to \infty$. This shows that    $\Gamma^k_{{\bf W}^{(g)}}(I))\to 0$ strongly as $k\to\infty$, which implies  that  ${\bf W}^{(g)}\in \cD^{pure}_{g_2^{-1}}(F^2(H_n))$.
The proof is complete.
\end{proof}

\begin{theorem}\label{Be-gen}  Let $f= 1+\sum_{ |\alpha|\geq 1} b_\alpha
Z_\alpha$ be an arbitrary  admissible free holomorphic function and set 
$$d:=\sup \left\{\frac{b_\alpha}{b_{g_i\alpha}}:\ \alpha\in \FF_n^+, i\in \{1,\ldots, n\}\right\}.
$$
Then 
$\cM\subset F^2(H_n)\otimes \cK$ is a joint invariant subspace under  $W_1^{(f)}\otimes I_\cK,\ldots, W_n^{(f)}\otimes I_\cK$ if and only if 
there is a Hilbert space $\cD$ and  a partial isometry $\Psi:F^2(H_n)\otimes \cD\to F^2(H_n)\otimes \cK$ such that
$$
\cM=\Psi(F^2(H_n)\otimes \cD)
$$
and 
$$
\sqrt{d}\Psi (S_i\otimes I_\cD)=(W_i^{(f)}\otimes I_\cK) \Psi,\qquad i\in \{1,\ldots, n\},
$$
where $S_1,\ldots, S_n$ are the left creation operators on the full Fock space $F^2(H_n)$.
\end{theorem}
\begin{proof}
 Using  Proposition \ref{inclusions}  and Theorem \ref{inv-subspace}, the proof  is similar to the proof of Corollary \ref{inv}.  
In this case, $g^{-1}=1-\frac{1}{d}(Z_1+\cdots +Z_n)$ and the universal model associated with the regular domain $\cD_{g^{-1}}$ is ${\bf W}^{(g)}=(\sqrt{d}S_1,\ldots, \sqrt{d}S_n)$.
\end{proof}

\bigskip

In what follows, we address the uniqueness of the Beurling-Lax-Halmos type representation of the joint invariant subspaces. We need some preliminary results.
 We recall that a   subspace $\cH\subseteq \cK$ is
called co-invariant under $\cS\subset B(\cK)$ if $X^*\cH\subseteq
\cH$ for any $X\in \cS$.

\begin{proposition}\label{cyclic}    Let $g= 1+\sum_{ |\alpha|\geq 1} b_\alpha
Z_\alpha$ be an admissible free holomorphic function with inverse $g^{-1}=1+\sum_{ |\gamma|\geq 1} a_\gamma
Z_\gamma$ and let ${\bf W}:=(W_1,\ldots, W_n)$ be  the corresponding  weighted left creation operators associated with $g$.
 If   
$\cM\subseteq F^2(H_n)\otimes \cK$ is  a co-invariant subspace under each operator
$W_i\otimes I_\cK$,  then  
\begin{equation*}
\overline{\text{\rm span}}\,\left\{\left(W_\alpha\otimes
I_\cK\right)\cM:\ \alpha\in \FF_n^+\right\}=F^2(H_n)\otimes \overline{(P_{\CC1}\otimes I_\cK) \cM}.
\end{equation*}
\end{proposition}

\begin{proof}
  Define the subspace 
 $\cE:=\overline{(P_{\CC1}\otimes I_\cK)\cM}\subset \cK$ and  let $\varphi\in \cM$, $\varphi\neq 0$,    with  Fourier representation
$\varphi=\sum_{\alpha\in \FF_n^+}e_\alpha\otimes h_\alpha$, 
$h_\alpha\in \cK.$
Let $\beta\in \FF_n^+$ be such that $h_\beta\neq 0$ and note that
$$
(P_{\CC1}\otimes I_\cK)(W_\beta^*\otimes I_\cK)\varphi=1\otimes
\frac{1}{\sqrt{b_\beta}} h_\beta.
$$
 Since $\cM$ is  a
co-invariant subspace under  each operator $W_i\otimes I_\cK$,
we deduce   that $h_\beta\in \cE$  and 
$$
(W_\beta\otimes I_\cK)(1\otimes
h_\beta)=\frac{1}{\sqrt{b_\beta}}e_\beta\otimes
h_\beta\in F^2(H_n)\otimes \cE, \quad \text{ for }\ \beta\in \FF_n^+.
$$
Consequentyly,   $
\varphi= \lim_{m\to\infty}\sum_{k=0}^m \sum_{|\alpha|=k}
 e_\beta\otimes h_\beta \in F^2(H_n)\otimes \cE,
$
which shows that  $\cM\subset F^2(H_n) \otimes \cE$ and
$$
\cY:= \overline{\text{\rm span}}\,\left\{\left(W_\alpha\otimes
I_\cK\right)\cM:\ \alpha\in \FF_n^+\right\}\subset F^2(H_n)\otimes \cE.
$$

To prove  the reverse inclusion, it is enough to  show  that $\cE\subset
\cY$.   To this end, let 
$h\in (P_{\CC1}\otimes I_\cK)\cM$ with  $h\neq 0$ and note that there exists $g\in \cM\subset
F^2(H_n)\otimes \cE$ such that $g=1\otimes
h+\sum\limits_{|\alpha|\geq 1} e_\alpha\otimes h_\alpha.$  According to Proposition \ref{W}, we have
$$
\Delta_{g^{-1}}({\bf W},{\bf W}^*):=\sum_{p=0}^\infty\sum_{ |\alpha|= p} a_\alpha  W_\alpha W_\alpha^*
$$
converges in the  strong operator topology and  $\Delta_{g^{-1}}({\bf W},{\bf W}^*)=P_{\CC1}$.
Consequently, 
\begin{equation*}
\begin{split}
1\otimes h&=(P_{\CC1}\otimes I_\cK) g 
=(\Delta_{g^{-1}}({\bf W},{\bf W}^*)\otimes I_\cK)g
\end{split}
\end{equation*}
 and, using the fact that  $\cM$ is co-invariant under  each operator $W_i\otimes I_\cK$,
we deduce that $h\in\cY$ for any $h\in (P_{\CC1}\otimes I_\cK)\cM$.
 Hence, we deduce that $(W_\alpha\otimes
I_\cK)(1\otimes \cE)\subset \cY$ for any $\alpha\in \FF_n^+$, which
implies
$
\frac{1}{\sqrt{b_\alpha}}e_\alpha\otimes \cE\subset
\cY$ for  all $\alpha\in \FF_n^+.$
Consequently, $F^2(H_n)\otimes \cE\subseteq \cY$, which completes the proof.
\end{proof}

Now, we can easily deduce the following result.

\begin{corollary}
 
 A subspace $\cM\subseteq F^2(H_n)\otimes \cK$ is reducing under each operator
 $W_i\otimes I_\cK$
if and only if   there exists a subspace $\cE\subseteq \cK$ such
that
\begin{equation*}
 \cM=F^2(H_n)\otimes \cE.
\end{equation*}
\end{corollary}

\bigskip

Let ${\bf W}=(W_1,\ldots, W_n)$ be   the corresponding  weighted left creation operators associated with an admissible free holomorphic function $f$.
 Fix an $n$-tuple  $Y:=(Y_1,\ldots, Y_n)\in B(\cK)^n$.
A bounded linear operator $A:F^2(H_n)\otimes \cH\to  \cK$  is called multi-analytic with respect to 
  ${\bf W}$ and
$Y$  if
$
A(W_i\otimes I_\cH)=Y_i A, i\in \{1,\ldots, n\}.
$
The support of $A$, denoted by  $\text{\rm supp}\, A$, is the smallest reducing subspace of $F^2(H_n)\otimes \cH$ under the operators $W_1\otimes I_\cH,\ldots, W_n\otimes I_\cH$, containing the co-invariant subspace
$\cM:=\overline{A^*(\cK)}$. Equivalently, $\text{\rm supp}\, A$ is the smallest reducing subspace $\cN$ under $W_1\otimes I_\cH,\ldots, W_n\otimes I_\cH$ such that $A|_{\cN^\perp}=0$.
According to Proposition \ref{cyclic}, we have 
$$
\text{\rm supp}\, A=\overline{\text{\rm span}}\,\left\{\left(W_\alpha\otimes
I_\cK\right)\cM:\ \alpha\in \FF_n^+\right\}=F^2(H_n)\otimes \cG,
$$
where $\cG:=\overline{(P_{\CC1}\otimes I_\cH)\left(A^*(\cK)\right)}$. Note that $A|_{\text{\rm supp}\, A}:F^2(H_n)\otimes \cG\to \cK$ is a multi-analytic operator with respect to ${\bf W}$ and $Y$, and 
$AA^*=(A |_{\text{\rm supp}\, A})(A |_{\text{\rm supp}\, A})^*$.

Let $A_1:F^2(H_n)\otimes \cK_1\to \cK$ and $A_2:F^2(H_n)\otimes \cK_2\to  \cK$ be multi-analytic operators with respect to ${\bf W}$ and $Y$. We say that $A_1$ and $A_2$  {\it coincide} if there is a unitary operator $\tau:\cK_1\to \cK_2$ such that $A_1=A_2(I_{F^2(H_n)}\otimes \tau)$.
 
\begin{theorem} \label{unique}  Let ${\bf W}=(W_1,\ldots, W_n)$ be   the corresponding  weighted left creation operators associated with an admissible free holomorphic function $f$.
Let $A_1:F^2(H_n) \otimes \cK_1\to \cK$ and $A_2:F^2(H_n)\otimes \cK_2\to  \cK$ be multi-analytic operators with respect to ${\bf W}$ and $Y=(Y_1,\ldots, Y_n)\in B(\cK)^n$,  such that
$A_1A_1^*=A_2A_2^*.
 $
 Then there is a partial isometry $V:\cK_1\to \cK_2$ such that
$
A_1=A_2(I\otimes V),
$
where $I\otimes V:F^2(H_n)\otimes \cK_1\to F^2(H_n)\otimes \cK_2$ is   a partially isometric multi-analytic operator with respect to ${\bf W}$, having the  initial space $\text{\rm supp}\, A_1$ and the final space $\text{\rm supp}\, A_2$. In particular, $A_1|_{\text{\rm supp}\, A_1}$ coincides with $A_2|_{\text{\rm supp}\, A_2}$.
\end{theorem}
\begin{proof} Suppose $g^{-1}=1+\sum_{ |\alpha|\geq 1} a_\alpha
Z_\alpha$  and recall  (see Poposition \ref{W}) that
$$
\Delta_{g^{-1}}({\bf W},{\bf W}^*):=\sum_{p=0}^\infty\sum_{ |\alpha|= p} a_\alpha  W_\alpha W_\alpha^*
=P_{\CC1},
$$
where the convergence is in the strong operator topology.
Since  $A_1A_1^*=A_2A_2^*$,    we obtain
 \begin{equation*}
 \begin{split}
 A_1(P_{\CC1}\otimes I_{\cK_1})A_1^*&= A_1\left(\sum_{p=0}^\infty\sum_{ |\alpha|= p} a_\alpha  W_\alpha W_\alpha^*\otimes I_{\cK_1}\right)A_1^*\\
 &=\sum_{p=0}^\infty\sum_{ |\alpha|= p} a_\alpha  Y_\alpha A_1A_1^*Y_\alpha^*
 =\sum_{p=0}^\infty\sum_{ |\alpha|= p} a_\alpha  Y_\alpha A_2A_2^*Y_\alpha^* \\
 &=
 A_2\left(\sum_{p=0}^\infty\sum_{ |\alpha|= p} a_\alpha  W_\alpha W_\alpha^*\otimes I_{\cK_2}\right)A_2^*
 =A_2(P_{\CC1}\otimes I_{\cK_2})A_2^*.
 \end{split}
 \end{equation*}
 Hence, we deduce that
 $$
 \left\|(P_{\CC1}\otimes I_{\cK_1})A_1^*x\right\|=\left\|(P_{\CC1}\otimes I_{\cK_2})A_2^*x\right\|,\qquad x\in \cK.
 $$
 Set $\cG_s:=\overline{(P_{\CC1}\otimes I_{\cK_s})\left(A_s^*(  \cK)\right)}$ for  $s=1,2$ and define the unitary operator $U:\cG_1\to \cG_2$ by
 $$
 U((P_{\CC1}\otimes I_{\cK_1})A_1^*x):=(P_{\CC1}\otimes I_{\cK_2})A_2^*x, \qquad  x\in \cK.
 $$
 This implies that there is a unique partial isometry $V:\cK_1\to \cK_2$ with initial space $\cG_1$ and final space $\cG_2$, extending $U$. Moreover, we have $A_1V^*=A_2|_{\cK_2}$, where we use the natural identification $\CC\otimes \cK_j=\cK_j\subset F^2(H_n)\otimes \cK_j$, $j=1,2$. Since $A_1, A_2$ are multi-analytic operators with respect to ${\bf W}$ and $Y$, we have
 \begin{equation*}
 \begin{split}
 A_1\left(I\otimes V^*\right) (W_\alpha\otimes I_{\cK_2}) (1\otimes h)
 &=Y_\alpha A_1 (I\otimes V^*)(1\otimes h)\\
 &=Y_\alpha A_2(1\otimes h)=A_2(W_\alpha\otimes I_{\cK_2})(1\otimes h)
 \end{split}
 \end{equation*}
for any $h\in \cK_2$ and $\alpha\in \FF_n^+$. Hence, taking into account that
  $$
  \Span \left\{ (W_\alpha\otimes I_{\cK_2})(\CC\otimes \cK_2): \  \alpha \in \FF_n^+\right\}=F^2(H_n)\otimes \cK_2,
  $$ 
  we deduce that $ A_1\left(I\otimes V^*\right)=A_2$, which implies $A_1=A_2(I\otimes V)$ and completes the proof.
\end{proof}

 Theorem \ref{inv-subspace}  shows that $\cM\subset F^2(H_n)\otimes \cK$ is a joint invariant subspace under  $W_1^{(f)}\otimes I_\cK,\ldots, W_n^{(f)}\otimes I_\cK$ if and only if there is a Hilbert space $\cD$
such that   $\cM=\Psi(F^2(H_n)\otimes \cD)$, where 
   $\Psi:F^2(H_n)\otimes \cD\to F^2(H_n)\otimes \cK$  is a  partial isometry such that
$$
\Psi (W_i^{(g)}\otimes I_\cD)=(W_i^{(f)}\otimes I_\cK) \Psi,\qquad i\in \{1,\ldots, n\}.
$$
Note that setting $\Psi_0:=\Psi|_{\text{\rm supp}\, \Psi}:F^2(H_n) \otimes \cG\to F^2(H_n)\otimes\cK$, where ${\text{\rm supp}\, \Psi}=F^2(H_n) \otimes \cG$, we also have
$$
\cM=\Psi_0\left(F^2(H_n) \otimes \cG\right) \quad \text{ and } \quad
\Psi_0 (W_i^{(g)}\otimes I_\cG)=(W_i^{(f)}\otimes I_\cK) \Psi_0
$$
and $\Psi_0$ is a partially isometric multi--analytic operator.
Therefore, all the joint invariant subspaces under $W_1^{(f)}\otimes I_\cK,\ldots, W_n^{(f)}\otimes I_\cK$ can be represented using partially isometric multi-analytic operators acting on their support. In this case, the representation turns out to be unique. Indeed, a simple application of
Theorem \ref{unique} yields the following result.

\begin{corollary} \label{unique2} Let $f$ and $g$ be two admissible free holomorphic functions and let 
$${\bf W}^{(f)}:=(W_1^{(f)},\ldots, W_n^{(f)})\quad  \text{ and } \quad  {\bf W}^{(g)}:=(W_1^{(g)},\ldots, W_n^{(g)})
$$
 be the  associted universal models. Assume that $\cD_{g^{-1}}$ is a regular domain and 
 $\cD^{pure}_{f^{-1}}\subset \cD^{pure}_{g^{-1}}.$
  The Beurling-Lax-Halmos type representation for  the joint invariant subspace  under $W_1^{(f)}\otimes I_\cK,\ldots, W_n^{(f)}\otimes I_\cK$ is essentially unique. More precisely, if
$$
\Psi_1\left(F^2(H_n)\otimes \cE_1\right)=\Psi_2\left(F^2(H_n)\otimes \cE_2\right),
$$
where $\Psi_j:F^2(H_n)\otimes \cE_j\to F^2(H_n)\otimes \cK$, $j=1,2$, are partially isometric multi-analytic operators, then there is a partial isometry $V:\cE_1\to \cE_2$ such that $\Psi_1=\Psi_2(I \otimes V)$. In particular,
$\Psi_1|_{\text{\rm supp}\, \Psi_1}$ coincides with $\Psi_2|_{\text{\rm supp}\, \Psi_2}$.
\end{corollary}

\bigskip

\section{Minimal dilations   for pure $n$-tuples  of operators and dilation index}

In this section, we study the minimal dilations of the pure $n$-tuples of operators in the noncommutative domain $\cD_{g^{-1}}(\cH)$. We show that, under certain conditions, the minimal dilation is unique. In addition, we introduce the dilation index and discuss its significance in dilation theory.

 Let $g$ be an admissible  free holomorphic function and let  ${\bf W}=(W_1,\ldots, W_n)$ be  the universal model associated with  $g$. 
 \begin{definition}
 Given  $T=(T_1,\ldots, T_n)\in \cD^{pure}_{g^{-1}}(\cH)$ and   a Hilbert space $\cG$, we say that the $n$-tuple  $(W_1\otimes I_\cG,\ldots, W_n\otimes I_\cG)$  is a dilation (coextension)  of  $T$  if 
$$
T_i^*=(W_i^*\otimes I_\cG)|_{\cH},\qquad i\in \{1,\ldots, n\},
$$
where $\cH$ is  identified with a coinvariant subspace  under each operators $W_i\otimes I_\cG$.   The dilation is called minimal if 
$$\overline{\text{\rm span}}\,\left\{\left(W_\alpha\otimes
I_\cG\right)\cH:\ \alpha\in \FF_n^+\right\}=F^2(H_n)\otimes \cG.
$$
The dilation index  of $T$, denoted by $d_T$, is the minimum dimension  of $\cG$ such that 
$(W_1\otimes I_\cG,\ldots, W_n\otimes I_\cG)$  is a dilation of  $T$.
\end{definition}
The first result gives some equivalent statements for the minimality condition.

\begin{proposition}\label{equiv}
Let $\cH\subseteq F^2(H_n)\otimes \cG$ be   a co-invariant subspace under
each operator $W_i\otimes I_\cG$. Then   the following statements are equivalent.
\begin{enumerate}
\item[(i)]
$\overline{\text{\rm span}}\,\left\{\left(W_\alpha\otimes
I_\cG\right)\cH:\ \alpha\in \FF_n^+\right\}=F^2(H_n)\otimes \cG$;
\item[(ii)] there is no proper subspace $\cN$  of $\cG$ such that $\cH\subset F^2(H_n)\otimes \cN$;

\item[(iii)] $ \overline{(P_{\CC1}\otimes I_\cG) \cH}=\cG$;
\item[(iv)] $\cH^\perp \bigcap \cG=\{0\}$.

\end{enumerate}
\end{proposition}
\begin{proof} Since the implication (i)$\implies$(ii) is obvious, we prove that (ii)$\implies$(i).
Assume that (ii) holds. According to  Proposition \ref{cyclic}, we have 
$$\overline{\text{\rm span}}\,\left\{\left(W_\alpha\otimes
I_\cG\right)\cH:\ \alpha\in \FF_n^+\right\}=F^2(H_n)\otimes\overline{P_\cG\cH},
$$
under the identification of  $\CC1\otimes \cG$ with $\cG$,
which shows that $\cH\subset F^2(H_n)\otimes\overline{P_\cG\cH}$. Consequently, we must have $\cG=\overline{P_\cG\cH}$. Therefore, items  (i) and (iii) hold. Note also that the implication (iii)$\implies$ (i) is due to Proposition \ref{cyclic}.

Now, we assume  that item (iii) holds and suppose that there is  $x\in \cH^\perp \bigcap \cG$ such that $x\neq 0$.  Since $\cH^\perp$ is invariant  under each operator $W_i\otimes I_\cG$, \ $i\in\{1,\ldots, n\}$, we deduce that
$F^2(H_n)\otimes \CC x\subset \cH^\perp$. This shows that $\cH\subset F^2(H_n)\otimes (\cG\ominus \CC x)$ and, consequently, $P_\cG \cH\subset  \cG\ominus \CC x$, which is a
 contradiction.
 
 Assume that (iv) holds and suppose that $ \overline{(P_{\CC1}\otimes I_\cG) \cH}\neq \cG$. Then there is $y\in \cG$, $y\neq 0$ such that $y\perp (P_{\CC1}\otimes I_\cG) \cH$. Hence, we have
 $\left<  y,h\right>
=\left< P_\cG y,h\right>=0$ for any $h\in \cH$, which implies $y\in \cH^\perp$. Therefore,
$y\in \cH^\perp \bigcap \cG=\{0\}$, which is a contradiction.  This shows that (iii) is equivalent to (iv).
\end{proof}

\begin{proposition}  \label{minimal} Let $T=(T_1,\ldots, T_n)\in \cD_{g^{-1}}(\cH)$ be a pure $n$-tuple and  assume that
$$
T_i^*=(W_i^*\otimes I_\cG)|_\cH,\qquad i\in\{1,\ldots, n\},
$$
where $\cG$ is a Hilbert space and $\cH$ is a co-invariant subspace  under  each operator $W_1\otimes I_\cG,\ldots, W_n\otimes I_\cG$.  Then
$$
\dim \overline{ \Delta_{g^{-1}}(T,T^*)\cH}\leq \dim\cG.
$$
If $(W_1\otimes I_\cG,\ldots, W_n\otimes I_\cG)$ is a minimal dilation of $T$, then 
$\dim \overline{ \Delta_{g^{-1}}(T,T^*)\cH}=d_T,
$
where $d_T$ is the dilation index.
\end{proposition}
\begin{proof} Since 
$$
\Delta_{g^{-1}}(T,T^*)\cH =P_\cH(\Delta_{g^{-1}}(W,W^*)\otimes I_\cG)|_\cH= P_\cH(P_{\CC1}\otimes I_\cG)|_\cH=
P_\cH P_\cG \cH,
$$
we deduce that $\dim \overline{ \Delta_{g^{-1}}(T,T^*)\cH}\leq \dim \cG$.

Now, assume that $(W_1\otimes I_\cG,\ldots, W_n\otimes I_\cG)$ is a minimal dilation of $T$.
According to Proposition \ref {equiv},         $ \overline{(P_{\CC1}\otimes I_\cG) \cH}=\cG$.
Hence, 
$\overline{ \Delta_{g^{-1}}(T,T^*)\cH}=\overline{P_\cH P_\cG\cH}=\overline{P_\cH\cG}$. Assume, by contradiction,  that
\begin{equation}\label{GG}
\dim \overline{P_\cH\cG} <\dim \cG.
\end{equation}
Then there is $x_0\in \cG$, $x_0\neq 0$ such that $P_\cH x_0=0$. Indeed, if $P_\cH x\neq 0$ for any $x\in \cG$, then the operator $P_\cH|_\cG: \cG\to \overline{P_\cH\cG}$ is one-to-one and, therefore, 
$\dim \overline{P_\cH\cG}\geq \dim\cG$, which contradicts  relation \eqref{GG}. Note that $x_0\in \cG\cap \cH^\perp$ and, due to Proposition \ref{equiv},  contradicts the fact that  $(W_1\otimes I_\cG,\ldots, W_n\otimes I_\cG)$ is a minimal dilation of $T$. Therefore, relation \eqref{GG} is false. Hence,
$$
\dim\cG\leq \dim \overline{P_\cH\cG} =\overline{ \Delta_{g^{-1}}(T,T^*)\cH}.
$$
On the other hand, the first part of the proof shows that  $\dim \overline{ \Delta_{g^{-1}}(T,T^*)\cH}\leq \dim\cG$.
The proof is complete.
\end{proof}

  According to Proposition \ref{minimal}, if the dilation is minimal then     the dilation index  is equal to $\dim \cG$. Is the converse true $?$    The next result shows that the answer is positive if  the dilation index   is finite. In addition,  we show how to get a minimal dilation from an arbitrary one.

\begin{proposition}  \label{many}  Let $T=(T_1,\ldots, T_n)\in \cD_{g^{-1}}(\cH)$ be a pure $n$-tuple with dilation index $d_T$  and  assume that
$$
T_i^*=(W_i^*\otimes I_\cG)|_\cH,\qquad i\in\{1,\ldots, n\},
$$
where $\cG$ is a Hilbert space and $\cH$ is a co-invariant subspace  under $W_1\otimes I_\cG,\ldots, W_n\otimes I_\cG$. 
If $d_T<\infty$, then the following statements hold.
\begin{enumerate}
\item[(i)] If $d_T<\dim \cG$, then there is a subspace $\cE\subset \cG$ such that
$$\dim(\cG\ominus \cE)=d_T, \quad F^2(H_n)\otimes \cE\subset \cH^\perp,
\quad 
 \text{ and  } \quad
T_i^*=(W_i^*\otimes I_{\cG\ominus \cE})|_\cH,\quad i\in\{1,\ldots, n\}.
$$
\item[(ii)]
$d_T=\dim\cG$ if and only if the dilation $(W_1\otimes I_\cG,\ldots, W_n\otimes I_\cG)$ is minimal.
\end{enumerate}
If $d_T=\infty$, then, taking $\cG_0:=\overline{(P_{\CC1}\otimes I_\cG) \cH}$, we have
$$\dim\cG_0=d_T, \quad F^2(H_n)\otimes (\cG\ominus \cG_0)\subset \cH^\perp,
\quad 
T_i^*=(W_i^*\otimes I_{\cG_0})|_\cH,\quad i\in\{1,\ldots, n\},
$$
and  $(W_1\otimes I_{\cG_0},\ldots, W_n\otimes I_{\cG_0})$ is a minimal dilation of $T$.
\end{proposition}
\begin{proof}
Assume that $d_T<\infty$. If $d_T<\dim \cG$, then Proposition \ref{minimal} shows that  $(W_1\otimes I_\cG,\ldots, W_n\otimes I_\cG)$ is not  a minimal dilation of $T$ and, due to Proposition \ref{equiv}, $\cH^\perp\cap \cG\neq \{0\}$.
Let $\cE$ be a maximal subspace of $\cG$ such that $\cE\subset \cH^\perp$. Since $\cH^\perp$ is invariant under $W_1\otimes I_\cG,\ldots, W_n\otimes I_\cG$, we have $F^2(H_n)\otimes \cE\subset \cH^\perp$. 
Consequently,
$$
T_i^*=(W_i^*\otimes I_{\cG\ominus \cE})|_\cH,\qquad i\in\{1,\ldots, n\}.
$$
If $d_T<\dim(\cG\ominus \cE)$, then, as above, we prove that $\cH^\perp \cap (\cG\ominus \cE)\neq \{0\}$ and, consequently, we can find a subspace $\cE'\supset \cE$, $\cE'\neq \cE$, such that $\cE'\subset \cH^\perp$, which contradicts the maximality of $\cE$. Therefore, we must have $\dim(\cG\ominus \cE)=d_T$. This proves item (i).

To prove item (ii), assume that $d_T=\dim\cG$. Then one can prove that $\cH^\perp \cap \cG=\{0\}$. 
Indeed, if we assume, by contradiction,  that  there is a non-zero vector $x\in \cH^\perp \cap \cG$, then $F^2(H_n)\otimes \CC x\subset \cH^\perp$ and $T_i^*=(W_i^*\otimes I_{\cG\ominus \CC x})|_\cH$ for $i\in\{1,\ldots, n\}$. Due to Proposition \ref{minimal}, we have $d_T\leq \dim \cG-1$, which is a contradiction. Now, using Proposition \ref{equiv}, we have  $\overline{(P_{\CC1}\otimes I_\cG) \cH}=\cG$.
Conversely, assume that $\overline{(P_{\CC1}\otimes I_\cG) \cH}=\cG$. Due to Corollary \ref{equiv}, $(W_1\otimes I_{\cG},\ldots, W_n\otimes I_{\cG})$ is a minimal dilation of $T$. Using Proposition \ref{minimal}, we deduce that
$d_T=\dim \cG$. This proves item (ii).

To prove the last part of this proposition, assume that $d_T=\infty$ and take  $\cG_0:=\overline{(P_{\CC1}\otimes I_\cG) \cH}$. Due to Theorem \ref{cyclic}, we have
\begin{equation}\label{span}
\overline{\text{\rm span}}\,\left\{\left(W_\alpha\otimes
I_\cG\right)\cH:\ \alpha\in \FF_n^+\right\}=F^2(H_n)\otimes \cG_0.
\end{equation}
Since $\cG\ominus \cG_0\subset \cH^\perp$ and $\cH^\perp$ is an  invariant subspace under the operators 
$W_1\otimes I_\cG,\ldots, W_n\otimes I_\cG$, we have 
$F^2(H_n)\otimes (\cG\ominus \cG_0)\subset \cH^\perp$ and 
$T_i^*=(W_i^*\otimes I_{\cG_0})|_\cH$ for all  $i\in\{1,\ldots, n\}$.
Note that relation \eqref{span} shows that  $(W_1\otimes I_{\cG_0},\ldots, W_n\otimes I_{\cG_0})$ is a minimal dilation of $T$. Using Proposition \ref{minimal}, we conclude that $\dim \cG_0=d_T$.
The proof is complete.
\end{proof}

We say that two $n$-tuples of operators $(T_1,\ldots, T_n)$, \
$T_i\in B(\cH)$, and $(T_1',\ldots, T_n')$, \ $T_i'\in B(\cH')$, are
unitarily equivalent if there exists a unitary operator $U:\cH\to
\cH'$ such that
$T_i=U^* T_i' U$   for  any  $ i=1,\ldots, n.
$
According to Theorem \ref{model}, if $T$ is  a pure $n$-tuple,  then  $(W_1\otimes I_\cD,\ldots, W_n\otimes I_\cD)$ is a dilation of $T$ implemented  by the 
 Berezin kernel  $K_{g,T}$, i.e. 
 \begin{equation}
 \label{dil-canonic}
 K_{g,T} T_i^*=(W_i^*\otimes I_\cD)K_{g,T},\qquad i\in \{1,\ldots, n\}.
 \end{equation}
 We call this dilation the {\it canonical dilation}  of the  pure $n$-tuple $T$.  One can prove that the canonical dilation is minimal.
 Indeed, note that  
$$
(P_{\CC1} \otimes I_\cD)K_{g,T}h=\lim_{m\to \infty}\sum_{k=0}^m \sum_{|\alpha|=k}
P_{\CC1}  e_\alpha\otimes \Delta_{g^{-1}}(T,T^*)
T_\alpha^*h,\quad h\in \cH,
$$
where $\cD:=\overline{\Delta_{g^{-1}}(T,T^*) \cH}$,
 and, consequently, $\overline{(P_{\CC1} \otimes I_\cD)K_{g,T}\cH}=\cD$.  Using now
Proposition \ref{cyclic} in the particular case when $\cM:=K_{g,T}\cH$,
 we deduce that   
\begin{equation*}
F^2(H_n)\otimes \cD=\bigvee_{\alpha\in \FF_n^+}
(W_\alpha\otimes I_{\cD}) K_{g,T}\cH,
\end{equation*}
which proves the minimality of the canonical dilation of $T$.
\begin{theorem}  \label{Uni}
 Let $T=(T_1,\ldots, T_n)\in \cD_{g^{-1}}(\cH)$ and   $T'=(T'_1,\ldots, T'_n)\in \cD_{g^{-1}}(\cH')$ be  pure $n$-tuples and let  $(W_1\otimes I_\cD,\ldots, W_n\otimes I_\cD)$  and  $(W_1\otimes I_{\cD'},\ldots, W_n\otimes I_{\cD'})$ be  the canonical   dilations of $T$ and $T'$, respectively. Then $T$ is unitarily equivalent to $T'$ if and only if there is a unitary operator $U:\cD\to \cD'$ such that 
$(I\otimes U) \cH=\cH' $and 
$$
(I\otimes U) (W_i\otimes I_\cD)=(W_i\otimes I_{\cD'}) (I\otimes U),\qquad i\in \{1,\ldots, n\}.
$$
\end{theorem}
\begin{proof} Assume that $V:\cH\to \cH'$ is a unitary operator such that $VT_i=T_i'V$ for $i\in \{1,\ldots, n\}$.
Then, we have
$V \Delta_{g^{-1}}(T,T^*)= \Delta_{g^{-1}}(T',T^{\prime *})V$
which induces a unitary operator 
$$
U:\overline{ \Delta_{g^{-1}}(T,T^*)\cH}\to \overline{ \Delta_{g^{-1}}(T',T^{\prime*})\cH'}
$$
Now, we show that the unitary operator  $I_{F^2(H_n)}\otimes U:=F^2(H_n)\otimes \overline{ \Delta_{g^{-1}}(T,T^*)\cH} \to F^2(H_n)\otimes \overline{ \Delta_{g^{-1}}(T',T^{\prime*})\cH'}$ is an extension of $V$ through the noncommutative Berezin  kernels which isometrically identify  $\cH$ with $K_{g,T} \cH$, and $\cH'$ with 
$K_{g,T'} \cH'$. Indeed,  using the definition of the Berezin  kernel,  we have
\begin{equation*}
\begin{split}
(I\otimes U) K_{g,T} h&=(I\otimes U)\left(\sum_{\alpha\in \FF_n^+} \sqrt{b_\alpha}
e_\alpha\otimes \Delta_{g^{-1}}(T,T^*)^{1/2} T_\alpha^* h\right)
=\sum_{\alpha\in \FF_n^+} \sqrt{b_\alpha}
e_\alpha\otimes V\Delta_{g^{-1}}(T,T^*)^{1/2} T_\alpha^* h\\
&=\sum_{\alpha\in \FF_n^+} \sqrt{b_\alpha}
e_\alpha\otimes \Delta_{g^{-1}}(T,T^*)^{1/2} VT_\alpha^* h
=\sum_{\alpha\in \FF_n^+} \sqrt{b_\alpha}
e_\alpha\otimes \Delta_{g^{-1}}(T,T^*)^{1/2} T_\alpha^{\prime*}V h\\
&=K_{g,T'} Vh
\end{split}
\end{equation*}
for any $h\in \cH$. This completes the proof.
\end{proof}

We denote by $C^*({\bf W})$ the $C^*$-algebra generated by $W_1,\ldots W_n$ and the  identity. 
We also consider the lexicographic order on the unital free semigroup   $\FF_n^+$. The next two results we proved in  \cite{Po-Noncommutative domains} and are needed in what follows.

\begin{theorem}  \label{irreducible} Let $g= 1+\sum_{ |\alpha|\geq 1} b_\alpha
Z_\alpha$ be a free holomorphic function in a neighborhood of the origin such that  $b_\alpha>0$ and 
$$
\sup_{\alpha\in \FF_n^+} \frac{b_\alpha}{b_{g_i \alpha}}<\infty,\quad \text{ for every }\ i\in \{1,\ldots, n\},
$$ 
and 
let   ${\bf W}=(W_1,\ldots, W_n)$ be  the   weighted left creation operators associated with $g$.
Then  the following statements hold.
\begin{enumerate}
\item[(i)] The $C^*$-algebra $C^*({\bf W})$ is irreducible.
\item[(ii)] The $n$-tuples $(W_1\otimes I_\cH,\ldots,
W_n\otimes I_\cH)$ and  $(W_1\otimes I_{\cH'},\ldots, W_n\otimes
I_{\cH'})$ 
  are unitarily equivalent if and only if   $\dim \cH=\dim \cH'$.
\item[(iii)] If  there is $i\in \{1,\ldots, n\}$ such that
$$
\lim_{\gamma\to \infty} \left(\frac{b_{g_p\gamma}}{b_{g_ig_p\gamma}}-\frac{b_\gamma}{b_{g_p\gamma}}\right) =0\quad  \text{for any} \quad  p\in \{1,\ldots, n\},
$$
then  the $C^*$-algebra $C^*({\bf W})$ contains all the compact operators in $B(F^2(H_n))$.  \end{enumerate}
\end{theorem}

We remark that Theorem \ref{irreducible} holds, in particular, for  the admissible free holomorphic functions   provided by Example \ref{Bergman} and Example \ref{Dirichlet}.

\begin{theorem} \label{compact}  Let $g= 1+\sum_{ |\alpha|\geq 1} b_\alpha
Z_\alpha$ be a free holomorphic function in a neighborhood of the origin such that  $b_\alpha>0$ and 
$$
\sup_{\alpha\in \FF_n^+} \frac{b_\alpha}{b_{g_i \alpha}}<\infty,\quad \text{ for every }\ i\in \{1,\ldots, n\},
$$ 
and 
let  ${\bf W}=(W_1,\ldots, W_n)$ be  the  weighted left creation operators associated with $g$.
  Assume that   the series $\sum_{k=0}^\infty\sum_{|\alpha|=k} a_\alpha W_\alpha W_\alpha^*$ is convergent in the operator norm topology, where $g^{-1}=1+\sum_{|\alpha|\geq 1} a_\alpha Z_\alpha$. Then  the following statements hold.
  \begin{enumerate}
  \item[(i)]
$\boldsymbol\cK(F^2(H_n))\subset \overline{\text{\rm span}}\{W_\alpha W_\beta^*:\ \alpha,\beta\in \FF_n^+\},
$
 where  $\boldsymbol\cK(F^2(H_n))$ stands for the ideal of all compact
operators  in $B(F^2(H_n))$. 
\item[(ii)] There is a unique minimal nontrivial  two-sided ideal  $\cJ$ of $C^*({\bf W})$. Moreover, 
$\cJ$ coincides  with $\boldsymbol\cK(F^2(H_n))$.

\item[(iii)] If, in addition,
$$
\lim_{\gamma \to \infty} \left(\frac{b_{g_p\gamma}}{b_{g_ig_p\gamma}}-\frac{b_\gamma}{b_{g_p\gamma}}\right) =0
$$
for every $i,p\in \{1,\ldots, n\}$,  then 
$$
C^*({\bf W})=\overline{\text{\rm span}}\{W_\alpha W_\beta^*:\ \alpha,\beta\in \FF_n^+\}.
$$
\end{enumerate}
\end{theorem}
 
 We remark that Theorem \ref{compact} holds, in particular, for  the admissible free holomorphic functions of Example \ref{Bergman}.

\begin{theorem} \label{unique-minimal}
Let  $g= 1+\sum_{ |\alpha|\geq 1} b_\alpha
Z_\alpha$ be  an admissible free holomorphic function such that the following conditions hold.
\begin{enumerate}
\item[(i)] For any $i,p\in \{1,\ldots, p\}$, 
$$
\lim_{\gamma\to \infty} \left(\frac{b_{g_p\gamma}}{b_{g_ig_p\gamma}}-\frac{b_\gamma}{b_{g_p\gamma}}\right) =0.
$$
\item[(ii)]   The series  $\sum_{|\alpha|\geq 1} a_\alpha W_\alpha W_\alpha^*$ is convergent in the operator norm, where $g^{-1}=1+\sum_{|\alpha|\geq 1} a_\alpha Z_\alpha$.

\end{enumerate}
Then any pure $n$-tuple in $\cD_{g^{-1}}(\cH)$ has a unique minimal dilation up to an isomorphism.
\end{theorem}

\begin{proof}
Let $(W_1\otimes I_\cG,\ldots, W_n\otimes I_\cG)$  be a  minimal dilation dilation of $T$, i.e.
 \begin{equation}
\label{another} 
VT_i^*=(W_i^*\otimes I_\cG)V,\qquad i\in\{1,\ldots, n\},
\end{equation}
where $V:\cH\to F^2(H_n) \otimes \cG$ is an isometry. Note that $V\cH$ is
co-invariant under each operator $W_i\otimes I_\cG$,  and 
\begin{equation*}
F^2(H_n)\otimes \cG=\bigvee_{\alpha\in \FF_n^+} (W_\alpha\otimes
I_{\cG}) V\cH.
\end{equation*} According  to Theorem 1.13 from  \cite{Po-Noncommutative domains} and Theorem \ref{compact}, there exists a  unique unital completely positive
linear map
$$\Psi:C^*({\bf W})=\overline{\Span}\{W_\alpha W_\beta^*:\ \alpha,\beta\in \FF_n^+\}\to B(\cH)$$
 such that
$\Psi(W_\alpha W_\beta^*)=T_\alpha T_\beta^*$, \ $\alpha, \beta\in
\FF_n^+$.    Set $\cD:=\overline{\Delta_{g^{-1}}(T,T^*) \cH}$ and consider the $*$-representations
\begin{equation*}
\begin{split}
\pi_1:C^*({\bf W})\to B(F^2(H_n)&\otimes \cD),\quad \pi_1(X)
:= X\otimes I_{\cD}, \text{ and }\ \\
\pi_2:C^*({\bf W})\to B(F^2(H_n)&\otimes \cG),\quad \pi_2(X):=
X\otimes I_{\cG}.
\end{split}
\end{equation*}
Due to relations   \eqref{dil-canonic}, \eqref{another}, and the
co-invariance of the subspaces $K_{g,T}\cH$ and $V\cH$ under
$W_i\otimes I_\cD$ and $W_i\otimes I_\cG$, respectively,  we have
$$
\Psi(X)=K_{g,T}^*\pi_1(X)K_{g,T}=V^*\pi_2(X)V\quad \text{ for }
\  X\in C^*({\bf W}).
$$
Note that $\pi_1$
and $\pi_2$ are  minimal Stinespring dilations of $\Psi$.  Since
they are unique \cite{St}, there exists a unitary operator $U:F^2(H_n)\otimes
\cD \to F^2(H_n)\otimes \cG$ such that
$
U(W_i\otimes I_\cD )=(W_i\otimes
I_\cG)U,  i\in\{1,\ldots,n\},
$
and $UK_{g,T}=V$.  Since $U$ is unitary, we also have
$
U(W_i^*\otimes I_\cD)=(W_i^*\otimes
I_\cG)U, i\in\{1,\ldots,n\}.
$
On the other hand, since  $C^*({\bf W})$ is irreducible (see
Theorem \ref{irreducible}),  we must have $U=I\otimes A$, where
$A\in B(\cD,\cG)$ is a unitary operator.
Therefore, $\dim \cD=\dim\cG$ and
$UK_{g,T}\cH=V\cH$, which proves that the two dilations are
unitarily equivalent.
   The proof is complete.
\end{proof}

 We remark that if $s\in [1,\infty)$ and 
 $$
g_s:=1+\sum_{k=1}^\infty\left(\begin{matrix} s+k-1 \\k \end{matrix}\right)(Z_1+\cdots  +Z_k)^k.
$$ 
then the hypothesis  of Theorem \ref{unique-minimal} is satisfied and, consequently, any pure $n$-tuple in $\cD_{g_s^{-1}}(\cH)$ has a unique minimal dilation up to an isomorphism.

 Now, we can
characterize   the pure  $n$-tuples of operators in $\cD_{g^{-1}}(\cH)$ having  dilation index $1$.

\begin{corollary}\label{rank1}  Assume that $g$ satisfies the hypothesis of Theorem \ref{unique-minimal}  and let $(W_1,\ldots, W_n)$ be the universal model associated with $g$.

 \begin{enumerate}
 \item[(i)]
If $\cM\subset F^2(H_n)$ is  a co-invariant  subspace under each operator  $W_i$,
 then
$$
T:=(P_\cM W_1|_\cM,\ldots, P_\cM W_n|_\cM) 
$$
is a pure    $n$-tuple of operators in $\cD_{g^{-1}}(\cM)$ such that
$d_T=1$. Conversely, if $T\in \cD_{g^{-1}}(\cM)$ is pure and  $d_T=1$, then $T$ has the form above.
\item[(ii)]
If $\cM'$ is another co-invariant subspace under   each operator $W_i$,
   which gives rise to an $n$-tuple   $T':=(P_{\cM'} W_1|_{\cM'},\ldots, P_{\cM'} W_n|_{\cM'}) $, then $T$
and $T'$ are unitarily equivalent if and only if $$\cM=\cM'.$$
\end{enumerate}
\end{corollary}
\begin{proof}
 To prove item (i), note that 
\begin{equation*}\begin{split}
\Delta_{g^{-1}}(T,T^*)=
P_\cM\Delta_{g^{-1}}(W,W^*)|\cM
=P_\cM P_{\CC1} |\cM,
\end{split}
\end{equation*}
where ${\bf W}:=(W_1,\ldots, W_n)$.
Hence, $\text{\rm rank}\, \Delta_{g^{-1}}(T,T^*)\leq 1$. Since ${\bf W}$ is pure in $\cD_{g^{-1}}(F^2(H_n))$, so is 
  $T\in \cD_{g^{-1}}(\cM)$ . This
also implies that $\Delta_{g^{-1}}({T,T^*})\neq 0$, so $\text{\rm rank}\,
\Delta_{g^{-1}}({T,T^*})\geq 1$. Consequently, we have $d_T=\rank \Delta_{g^{-1}}({T,T^*})=1$.
The converse  follows from Theorem \ref{model} and Proposition \ref{minimal}.

To prove  item (ii), note that,  as in the proof of  Theorem
\ref{unique-minimal}, one can show that $T$ and $T'$ are unitarily
equivalent if and only if there exists a unitary operator
$\Lambda:F^2(H_n)\to F^2(H_n)$ such that
$$
\Lambda W_i=W_i \Lambda, \ \text{ for any  } i\in\{1, \ldots, n\}, \ \text{ and
}  \ \Lambda \cM=\cM'.
$$
Hence $\Lambda W_i^*=W_i^* \Lambda$, $i\in\{1,\ldots,n\}$. Since
$C^*({\bf W})$ is irreducible (see Theorem \ref{compact}),
$\Lambda$ must be a scalar multiple of the identity. Hence, we have
$\cM=\Lambda \cM=\cM'$.
The proof is complete.
\end{proof}

\bigskip

\section{Dilation theory on noncommutative domains}

This section is devoted to dilation theory for not necessarily pure  $n$-tules of operators  in noncommutative domains. The uniqueness of the dilation is also addressed.

 We refer the reader to \cite{Pa-book}, \cite{Pi-book} for basic facts concerning completely contractive (resp. positive) maps. We recall from  \cite{Po-Noncommutative domains} the following result that is needed in what follows.
\begin{theorem}\label{pure}
If   $g$ is an admissible  free holomorphic function and $T\in \overline{\cD^{pure}_{g^{-1}}(\cH)}$, then there is a unital completely contractive linear map
$$
\Psi_T:\cS:=\overline{span}\{W_\alpha W_\beta^*: \ \alpha,\beta\in \FF_n^+\}\to B(\cH),
$$
such that $\Psi_T (W_\alpha W_\beta^*)=T_\alpha T_\beta^*$  for any $\alpha,\beta\in \FF_n^+$.
If $T\in B(\cH)^n$ is a radially pure  $n$-tuple with respect to $\cD_{g^{-1}}(\cH)$, then
$$
\Psi_T(A)=\lim_{r\to 1}  K_{g,rT}^*(A\otimes I) K_{g,rT},\qquad A\in \cS,
$$
where the limit is in the operator norm.
\end{theorem}
 We remark  the map $\Psi_T$ in Theorem \ref{pure} is also completely positive  due to the fact that $\cS$ is an operator system.

 Let $C^*(\cY)$ be the $C^*$-algebra generated by a set of
  operators
$\cY\subset B(\cK)$ and the identity. A  subspace $\cH\subseteq \cK$
is called $*$-cyclic  for $\cY$    if
$$\cK=\overline{\text{\rm span}}\left\{Xh:\ X\in C^*(\cY), \ h\in
\cH\right\}.
$$

Here is  a $C^*$-algebra  version  of the Wold decomposition \cite{W}  for  the unital 
$*$-representations  of  the $C^*$-algebra $C^*({\bf W})$ which was obtained in  \cite{Po-Noncommutative domains}. 

\begin{theorem}  \label{wold} Let $g$ be an admissible free holomorphic function  and 
let  ${\bf W}=(W_1,\ldots, W_n)$ be  the  weighted left creation operators associated with $g$.
  If the series $\sum_{k=0}^\infty\sum_{|\alpha|=k} a_\alpha W_\alpha W_\alpha^*$ is convergent in the operator norm topology, where $g^{-1}=1+\sum_{|\alpha|\geq 1} a_\alpha Z_\alpha$, and   
$\pi:C^*({\bf W})\to B(\cK)$ is  a unital
$*$-representation  of $C^*({\bf W})$ on a separable Hilbert
space  $\cK$, then $\pi$ decomposes into a direct sum
$$
\pi=\pi_0\oplus \pi_1 \  \text{ on  } \ \cK=\cK_0\oplus \cK_1,
$$
where $\pi_0$ and  $\pi_1$  are disjoint representations of
$C^*({\bf W})$ on the Hilbert spaces
\begin{equation*}
\begin{split}
\cK_0:&=\overline{\text{\rm span}}\left\{\pi(W_\beta)
\Delta_{g^{-1}}(\pi({\bf W}),\pi({\bf W})^*)\cK:\ \beta\in
\FF_n^+\right\}\quad \text{ and } \\
 \cK_1:&=\cK\ominus \cK_0,
\end{split}
\end{equation*}
 respectively, where $\pi({\bf W}):=(\pi(W_1),\ldots, \pi(W_n))$. Moreover, up to an isomorphism,
\begin{equation*}
\cK_0\simeq F^2(H_n)\otimes \cG, \quad  \pi_0(X)=X\otimes I_\cG \quad
\text{ for } \  X\in C^*({\bf W}),
\end{equation*}
 where $\cG$ is a Hilbert space with
$$
\dim \cG=\dim \left[\text{\rm range}\,\Delta_{g^{-1}}(\pi({\bf W}),\pi({\bf W})^*)\right],
$$
 and $\pi_1$ is a $*$-representation  which annihilates the compact operators  in $C^*({\bf W})$  and
$$
\Delta_{g^{-1}}(\pi_1({\bf W}),\pi_1({\bf W})^*)=0.
$$
If $\pi'$ is another $*$-representation  of $C^*({\bf W})$  on a separable Hilbert space $\cK'$, then $\pi$ is unitarily equivalent to  $\pi'$ if and only if $\dim \cG=\dim \cG'$ and $\pi_1$ is unitarily equivalent to $\pi_1'$.
 \end{theorem}

A unital $*$-representation $\pi$  of the $C^*$-algebra  $C^*({\bf W})$ such that 
  $
  \Delta_{g^{-1}}(\pi({\bf W}),\pi({\bf W}^*))=0,
  $
is called {\it Cuntz type representation}.

\begin{theorem}\label{dilation}  Let $g$ be an admissible free holomorphic function  with
 $g^{-1}=1+\sum_{|\alpha|\geq 1} a_\alpha Z_\alpha$ and 
let  ${\bf W}=(W_1,\ldots, W_n)$ be  the  weighted left creation operators associated with $g$. Assume that  
 the series $\sum_{k=0}^\infty\sum_{|\alpha|=k} a_\alpha W_\alpha W_\alpha^*$ converges in the operator norm topology.

If $T=(T_1,\ldots, T_n)\in \overline{\cD^{pure}_{g^{-1}}(\cH)}$, then there is a $*$-representation  $\pi:C^*({\bf W})\to \cK_\pi$ on a separable Hilbert space $\cK_\pi$ satisfying the following properties.
\begin{enumerate}
\item[(i)]
$\pi$  annihilates  the compact operators  in $C^*({\bf W})$ and 
$
\Delta_{g^{-1}}(\pi({\bf W}),\pi({\bf W})^*)=0,
$
where $\pi({\bf W}):=(\pi(W_1),\ldots, \pi(W_n))$.
\item[(ii)]  $\cH$ can be identified with a  $*$-cyclic  co-invariant subspace  of $\cK:=(F^2(H_n)\otimes \cD)\bigoplus \cK_\pi$
under the operators 
$(W_i\otimes I_\cD)\bigoplus \pi(W_i)$ such that
$$
T_i^*=[(W_i^*\otimes I_\cD)\bigoplus \pi(W_i)^*] |_\cH,\qquad  i\in\{1,\ldots, n\},
$$
where $\cD:=\overline{\Delta_{g^{-1}}(T,T^*)\cH}$. Moreover, $T\in \cD_{g^{-1}}(\cH)$.
\end{enumerate}
\end{theorem}
\begin{proof}

According to  Theorem \ref{pure},   there is a unique unital
completely contractive linear map
$$
\Phi_T:\cS:=\overline{\Span}\{W_\alpha W_\beta^*: \ \alpha,\beta\in \FF_n^+\}\to B(\cH),
$$
such that $\Phi_T (W_\alpha W_\beta^*)=T_\alpha T_\beta^*$  for any $\alpha,\beta\in \FF_n^+$.
 Applying Arveson extension theorem
\cite{Arv1-acta} to the map $\Phi_T$, we find a unital
completely positive linear map $\widetilde\Phi_T:C^*({\bf W})\to
B(\cH)$ such the $\widetilde\Phi_T |_\cS=\Phi_T$. Let $\tilde\pi:C^*({\bf W})\to B(\tilde\cK)$ be a
minimal Stinespring dilation \cite{St} of $\widetilde\Phi_T$. Then
$$\widetilde\Phi_T(X)=P_{\cH} \tilde\pi(X)|\cH,\quad X\in C^*({\bf W}),
$$
and $\tilde\cK=\overline{\text{\rm span}}\{\tilde\pi(X)h:\ h\in
\cH\}.$  Since, for each $i\in \{1,\ldots, n\}$,
\begin{equation*}
\begin{split}
\widetilde\Phi_T(W_i W_i^*)&=T_iT_i^*=P_\cH\tilde\pi(W_i)\tilde\pi(W_i^*)|\cH\\
&=P_\cH \tilde\pi(W_i)(P_\cH+P_{\cH^\perp})\tilde\pi(W_i^*)|\cH\\
&=\widetilde\Phi_T(W_i W_i^*)+ (P_\cH
\tilde\pi(W_i)|_{\cH^\perp})(P_{\cH^\perp} \tilde\pi(W_i^*)|\cH),
\end{split}
\end{equation*}
 it is easy to see  that $P_\cH \tilde\pi(W_i)|_{\cH^\perp}=0$ and
\begin{equation}\label{morph}
\begin{split}
\widetilde\Phi_T(W_\alpha X)&=P_\cH(\tilde\pi(W_\alpha) \tilde\pi(X))|\cH\\
&=(P_\cH\tilde\pi(W_\alpha)|\cH)(P_\cH\tilde\pi(X)
|\cH)
=\widetilde\Phi_T(W_\alpha) \widetilde\Phi_T(X)\end{split}
\end{equation}
for any $X\in C^*({\bf W})$ and $\alpha\in \FF_n^+$. Since
 $P_\cH \tilde\pi(W_i)|_{\cH^\perp}=0$,
  $\cH$ is an
invariant subspace under each operator $\tilde\pi(W_i)^*$.
 Consequently, we have
  \begin{equation}\label{coiso}
\tilde\pi(W_i)^*|_\cH=\widetilde\Phi_T (W_i^*)=T_i^*,\quad i\in \{1,\ldots, n\}.
\end{equation}

 By the Wold decomposition of Theorem \ref{wold},
the representation $\tilde\pi$ decomposes into a direct sum
$\tilde\pi=\pi_0\oplus \pi$ on $\tilde \cK=\cK_0\oplus \cK_\pi$,
where $\pi_0$, $\pi$  are disjoint representations of
$C^*({\bf W})$ on the Hilbert spaces $\cK_0$ and $\cK_\pi$,
respectively, such that
\begin{equation}\label{sime}
\cK_0\simeq F^2(H_n)\otimes \cG, \quad  \pi_0(X)=X\otimes I_\cG, \quad
X\in C^*({\bf W}),
\end{equation}
 for some Hilbert space $\cG$, and $\pi$ is a representation such that
$\pi(\boldsymbol\cK(F^2(H_n)))=0$,  where $\boldsymbol\cK(F^2(H_n))$ stands for all compact operators in $B(F^2(H_n))$. Now, taking into account that $\Delta_{g^{-1}}({\bf W}, {\bf W}^*)=P_{\CC 1}$
  is a
  rank-one projection   in $C^*({\bf W})$, we  have that
$$
\Delta_{g^{-1}}(\pi({\bf W}),\pi({\bf W})^*)=0\quad \text{ and }\quad
\dim \cG=\dim (\text{\rm range}\,\tilde\pi(P_{\CC 1})).
$$
Taking into account that  the  Stinespring representation $\tilde\pi$  is minimal and
using the proof of Theorem \ref{compact}, we have 
\begin{equation*}\begin{split}
\text{\rm range}\,\tilde\pi(P_\CC)&=
\overline{\text{\rm span}}\{\tilde\pi(P_{\CC1})\tilde\pi(X)h:\ X\in C^*({\bf W}), h\in \cH\}\\
&=
\overline{\text{\rm span}}\{\tilde\pi(P_{\CC1})\tilde\pi(Y)h:\ Y\in \boldsymbol\cK(F^2(H_n)), h\in \cH\}\\
&=
\overline{\text{\rm span}}\{\tilde\pi(P_{\CC1})\tilde\pi(W_\alpha P_{\CC1} W_\beta^*)h:\ \alpha,\beta\in \FF_n^+, h\in \cH\}\\
&= \overline{\text{\rm
span}}\{\tilde\pi(P_{\CC1})\tilde\pi(W_\beta^*)h:\ \beta\in
\FF_n^+, h\in \cH\}.
\end{split}
\end{equation*}
Using   relation \eqref{morph}, we deduce that
\begin{equation*}\begin{split}
\left<\tilde\pi(P_{\CC1})\tilde\pi(W_\alpha^*)h,
\tilde\pi(P_{\CC1})\tilde\pi(W_\beta^*)k\right>
&=
\left<h,\pi(W_\alpha)\pi(P_{\CC1})\pi(W_\beta^*)k\right>\\
&= \left<h,T_\alpha \Delta_{g^{-1}}(T,T^*)T_\beta^*k\right>\\
&= \left< \Delta_{g^{-1}}(T,T^*)^{1/2}T_\alpha^*h, \Delta_{g^{-1}}(T,T^*)^{1/2}T_\beta^*k\right>
\end{split}
\end{equation*}
for any $h, k \in \cH$. Consequently, there is   a unitary operator
$\Lambda:\text{\rm range}\,\tilde\pi(P_{\CC1})\to
\overline{ \Delta_{g^{-1}}(T,T^*)\cH}$  defined by
$$
\Lambda(\tilde\pi(P_{\CC1})\tilde\pi(W_\alpha^*)h):= \Delta_{g^{-1}}(T,T^*)^{1/2}
T_\alpha^*h,\quad h\in \cH, \,\alpha\in \FF_n^+.
$$
 This shows that
$$
\dim[\text{\rm range}\,\pi(P_{\CC1})]= \dim
\overline{ \Delta_{g^{-1}}(T,T^*)\cH}=\dim \cG.
$$
The required dilation is obtained using relations \eqref{coiso} and
\eqref{sime},  and identifying    $\cG$ with
$\overline{ \Delta_{g^{-1}}(T,T^*)\cH}$. Since $T_i^*=[(W_i^*\otimes I_\cD)\bigoplus \pi(W_i)^*]_\cH$ for $i\in\{1,\ldots, n\}$, it is easy to see that $T\in \cD_{g^{-1}}(\cH)$.
 The proof is complete.
\end{proof}
Note  that any   radial-pure $n$-tuple $T$  in $\cD_{g^{-1}}(\cH)$ is in  $\overline{\cD^{pure}_{g^{-1}}(\cH)}$, and, consequently,  Theorem \ref{pure} applies.

\begin{corollary}  Let  $\cD_{g^{-1}} $ be  a radially pure domain  such   that $g^{-1}=\sum_{\alpha\in \FF_n^+} a_\alpha Z_\alpha$  has the property  that the series $\sum_{k=0}^\infty\sum_{|\alpha|=k} a_\alpha W_\alpha W_\alpha^*$ converges in the operator norm,  and let $T=(T_1,\ldots, T_n)\in B(\cH)^n$.  Then 
$T\in  \cD_{g^{-1}}(\cH)$  if and only if  there is a Hilbert space $\cD$ and  a  Cuntz $*$-representation  $\pi:C^*({\bf W})\to \cK_\pi$ on a separable Hilbert space $\cK_\pi$    such that 
$$
T_i^*=[(W_i^*\otimes I_\cD)\bigoplus \pi(W_i)^*] |_\cH,\qquad  i\in\{1,\ldots, n\}.
$$
\end{corollary} 
\begin{proof} Since  $\cD_{g^{-1}} $ is   a radially pure domain, we have $\cD_{g^{-1}}(\cH)\subset\overline{\cD^{pure}_{g^{-1}}(\cH)}$.  The direct implication  follows from Theorem \ref{dilation}.
To prove the converse, assume that $T=(T_1,\ldots, T_n)\in B(\cH)^n$ satisfies the relation above.
Since ${\bf W}\in \cD_{g^{-1}}(F^2(H_n))$ and $\pi({\bf W})\in  \cD_{g^{-1}}(\cK_\pi)$, we conclude that $T\in \cD_{g^{-1}}(\cH)$. The proof is complete.
\end{proof} 

We remark that,  in the particular case when 
$g^{-1}=(1-\sum_{1\leq |\alpha|\leq N} c_\alpha  Z_\alpha )^m$, $ N, m\geq 1,
$ 
and 
$c_\alpha\geq 0$  with $c_{g_i}>0$ for $i\in \{1,\ldots, n\}$,  the noncommutative set $\cD_{g^{-1}} $ is   a radially pure domain. This domain   was  studied in \cite{Po-Berezin}.

\begin{corollary}\label{unique3}
Assume  the hypotheses of Theorem $\ref{dilation}$ and that
$$
\lim_{\gamma\to \infty} \left(\frac{b_{g_p\gamma}}{b_{g_ig_p\gamma}}-\frac{b_\gamma}{b_{g_p\gamma}}\right) =0
$$
for every $i,p\in \{1,\ldots, n\}$.  
 Then the  dilation of Theorem $\ref{dilation}$ is minimal, i.e.,
$\cK=\bigvee\limits_{\alpha\in \FF_n^+} V_\alpha \cH$,   and
it is unique up to a unitary equivalence.
\end{corollary}

\begin{proof} According to Theorem \ref{compact}, part (iii), we have 
$
C^*({\bf W})=\overline{\text{\rm span}}\{W_\alpha W_\beta^*:\ \alpha,\beta\in \FF_n^+\}.
$
Consequently, the map $\widetilde\Phi_T$
in the proof of  Theorem \ref{dilation} is unique. The uniqueness of the
minimal Stinespring representation  \cite{St} implies  the
uniqueness of the minimal dilation of Theorem \ref{dilation}.
\end{proof}

Using standard arguments  concerning
representation theory of $C^*$-algebras \cite{Arv-book}, one can use Theorem \ref{dilation} to obtain
 the following result.

\begin{remark} Under the hypotheses of Corollary
\ref{unique3}, let $T':=(T_1',\ldots, T_n')\in  \overline{\cD^{pure}_{g^{-1}}(\cH')}$   and
let $V':=(V_1',\ldots, V_n')$ be the corresponding 
dilation. Then  $T$ and $T'$ are unitarily equivalent
if and only if
$$\dim \overline{\Delta_{g^{-1}}(T,T^*)\cH}=\dim \overline{\Delta_{g^{-1}}(T',{T'}^*)\cH'}
$$
and there are unitary operators
 $U:\overline{\Delta_{g^{-1}}(T,T^*)\cH}\to \overline{\Delta_{g^{-1}}(T',{T'}^*)\cH'}$
  and $\Gamma:\cK_\pi\to \cK_{\pi '}$
such that $ \Gamma\pi(W_i)=\pi'(W_i) \Gamma$  for  all  
$i\in\{1,\ldots, n\}$, and
$$\left[ \begin{matrix} I_{F^2(H_n)}\otimes U&0\\
0&\Gamma\end{matrix} \right] \cH=\cH'.
$$
\end{remark}

\begin{corollary} \label{part-cases} Assume the hypotheses of Corollary
\ref{unique3}.   Let  $T=(T_1,\ldots,T_n)\in T\in \overline{\cD^{pure}_{g^{-1}}(\cH)}$ and 
let $V:=(V_1,\ldots, V_n)$ be its minimal dilation on the Hilbert space  $\cK:=(F^2(H_n)\otimes \cD)\bigoplus \cK_\pi$.
   Then the following statements hold.
\begin{enumerate}
\item[(i)]
 $V$ is a pure $n$-tuple  in $\cD_{g^{-1}}(\cK)$ if and only if
 $T $
 is  a pure $n$-tuple in $\cD_{g^{-1}}(\cH)$;
\item[(ii)] $V$ is a Cuntz $n$-tuple  in $\cD_{g^{-1}}(\cK)$, i.e. $\Delta_{g^{-1}}(V,V^*)=0$,  if and only if  $\Delta_{g^{-1}}(T,T^*)=0$.
   \end{enumerate}
\end{corollary}
\begin{proof}
Note that
\begin{equation*}
\begin{split}
\sum_{\alpha \in \FF_n^+} b_\alpha T_\alpha \Delta_{g^{-1}}(T,T^*) T_\alpha^*
&=
P_\cH\sum_{\alpha \in \FF_n^+} b_\alpha V_\alpha \Delta_{g^{-1}}(V, V^*) V_\alpha^*|_\cH=P_\cH \left[\begin{matrix}
 I_{F^2(H_n)\otimes \cD}&0\\0& 0\end{matrix}\right]|\cH.
\end{split}
\end{equation*}
Consequently,  $T$ is  a pure $n$-tuple 
 if and only if $P_\cH P_{F^2(H_n)\otimes \cD} |\cH=I_\cK$. The later
condition is equivalent to    $\cH\subset F^2(H_n)\otimes \cD $. On the
other hand, since $F^2(H_n)\otimes \cD $ is
reducing for  $V_1,\ldots, V_n$, and $ \cK$ is the smallest
reducing subspace for   $V_1,\ldots, V_n$, which contains $\cH$, we
must have $\cK=F^2(H_n)\otimes \cD $.
 To prove item (ii), note that
 $$ \Delta_{g^{-1}}(V,V^*)=
\left[\begin{matrix} \Delta_{g^{-1}}({\bf W},{\bf W}^*)\otimes I_\cD&0\\0& 0\end{matrix}\right]
$$
and, consequently,  $ \Delta_{g^{-1}}(V,V^*)=0$ if and only  if 
$\Delta_{g^{-1}}({\bf W},{\bf W}^*)\otimes I_\cD=0$. 
Since  $\Delta_{g^{-1}}({\bf W},{\bf W}^*)=P_{\CC1}$, the later relation is equivalent to $\cD=\{0\}$.
Therefore, $ \Delta_{g^{-1}}(V,V^*)=0$ if and only  if $\Delta_{g^{-1}}(T,T^*)=0$. This   completes the proof.
\end{proof}

\bigskip

\section{Commutant lifting  and Toeplitz-corona theorems}

The goal of this section is to prove a commutant lifting theorem  and use it to provide factorization results for multi-analytic operators and a Toeplitz-corona theorem in our setting.

Let $f$ and $g$ be two admissible free holomorphic functions and let 
$${\bf W}^{(f)}:=(W_1^{(f)},\ldots, W_n^{(f)})\quad  \text{ and } \quad  {\bf W}^{(g)}:=(W_1^{(g)},\ldots, W_n^{(g)})
$$
 be the universal models associated with the noncommuative domains $\cD_{f^{-1}}$ and $\cD_{g^{-1}}$, respectively. 
We assume that $\cD^{pure}_{f^{-1}}\subset \cD^{pure}_{g^{-1}}$.
 Similarly to the proof of  Theorem \ref{inclusion}, one can show that  given  a Hilbert space $\cE_*$,
  there is a Hilbert space $\cF$ and an isometry $V:F^2(H_n)\otimes \cE_*\to F^2(H_n)\otimes \cF$ such that 
\begin{equation}\label{V}
V(W_i^{(f)*}\otimes I_{\cE_*})=(W_i^{(g) *}\otimes I_\cF)V,\qquad i\in \{1,\ldots, n\}.
\end{equation}
We need this result in what follows.
\begin{proposition} \label{reduction}
Let $\cG_1\subset F^2(H_n)\otimes \cE$ be a co-invariant subspace under  each operator 
$W_i^{(g) }\otimes I_\cE$ and let  $\cG_2\subset F^2(H_n)\otimes \cE_*$ be a co-invariant subspace under  each operator 
$W_i^{(f) }\otimes I_{\cE_*}$. Let $X\in B(\cG_1,\cG_2)$ and define $\tilde X:=(V|_{\cG_2})X:\cG_1\to V(\cG_2)$.  

If the $n$-tuples  $T=(T_1,\ldots, T_n)$ and  $T'=(T_1',\ldots, T_n')$ are defined by
$$
T_i:=P_{\cG_2}(W_i^{(f)}\otimes I_{\cE_*})|_{\cG_2} \ \text{ and }\  T_i':=P_{V(\cG_2)}(W_i^{(g)}\otimes I_{\cF})|_{V(\cG_2)}, 
$$
then the following statements hold.
\begin{enumerate}

\item[(i)]  $ X[P_{\cG_1}(W_i^{(g)}\otimes I_{\cE})|_{\cG_1}]=T_iX$ for every $i\in \{1,\ldots, n\}$ if and only if 
$$ \tilde X[P_{\cG_1}(W_i^{(g)}\otimes I_{\cE})|_{\cG_1}]=T_i' \tilde X\quad \text{ for every }\quad  i\in \{1,\ldots, n\}.
$$

\item[(ii)] $\|X\|\leq 1$ if and only if $\|\tilde X\|\leq 1$.
\end{enumerate}
\end{proposition}
\begin{proof}
Since $\cG_2\subset F^2(H_n)\otimes \cE_*$ is a co-invariant subspace under  each operator 
$W_i^{(f) }\otimes I_{\cE_*}$ and  $V|_{\cG_2}$  is an isometry, relation \eqref{V} implies 
$$
(W_i^{(f)*}\otimes I_{\cE_*})|_{\cG_2}=(V|_{\cG_2})^*(W_i^{(g) *}\otimes I_\cF)(V|_{\cG_2}),\qquad i\in \{1,\ldots, n\}.
$$
Taking the adjoint of the later relation and multiplying to the left by $V|_{\cG_2}$, we obtain
$$
(V|_{\cG_2})(W_i^{(f)}\otimes I_{\cE_*})|_{\cG_2}=P_{V(\cG_2)}(W_i^{(g)}\otimes I_\cF)(V|_{\cG_2}).
$$
Consequently, 
\begin{equation}\label{new}
P_{V(\cG_2)}(W_i^{(g)}\otimes I_\cF)(V|_{\cG_2})\tilde X=(V|_{\cG_2})P_{\cG_2}(W_i^{(f)}\otimes I_{\cE_*})|_{\cG_2}X.
\end{equation}
On the other hand, we have
$$
\tilde X [P_{\cG_1}(W_i^{(g)}\otimes I_{\cE})|_{\cG_1}]=(V|_{\cG_2})X [P_{\cG_1}(W_i^{(g)}\otimes I_{\cE})|_{\cG_1}].
$$
Note that if $ X[P_{\cG_1}(W_i^{(g)}\otimes I_{\cE})|_{\cG_1}]=T_iX$, then relation \eqref{new} implies 
$\tilde X[P_{\cG_1}(W_i^{(g)}\otimes I_{\cE})|_{\cG_1}]=T_i' \tilde X.$
Conversely, if the later relation holds,  we can use again relation \eqref{new} to deduce that
$$
(V|_{\cG_2})X [P_{\cG_1}(W_i^{(g)}\otimes I_{\cE})|_{\cG_1}]=P_{V(\cG_2)}(W_i^{(g)}\otimes I_\cF)(V|_{\cG_2})\tilde X
=(V|_{\cG_2}) [P_{\cG_2}(W_i^{(f)}\otimes I_{\cE_*})|_{\cG_2}] X.
$$
Since $X\in B(\cG_1,\cG_2)$ and $V|_{\cG_2}$  is an isometry, we obtain
$$
X [P_{\cG_1}(W_i^{(g)}\otimes I_{\cE})|_{\cG_1}]=[P_{\cG_2}(W_i^{(f)}\otimes I_{\cE_*})|_{\cG_2}] X,
$$
which completes the proof of part (i).   Since $\tilde X:=(V|_{\cG_2})$, part (ii) is obvious. The proof is complete.
\end{proof}

In \cite{Po-Noncommutative domains}, we introduced the noncommutative Hardy algebras $F^\infty(g)$  and $R^\infty(g)$ associated with formal power series $g$ with strictly positive coefficients  and provided a $w^*$-continuous $F^\infty(g)$-functional calculus for pure (resp. completely non-coisometric) $n$-tuples of operators in the noncommutative domain $\cD_{g^{-1}}(\cH)$.  We recall that  the noncommutative  Hardy space $F^\infty(g)$  satisfies the relation
  $$
  F^\infty(g)=\overline{\cP({\bf W})}^{SOT}=\overline{\cP({\bf W})}^{WOT}=\overline{\cP({\bf W})}^{w*},
  $$
where $\cP({\bf W})$ stands for the algebra of all polynomials in $W_1,\ldots, W_n$ and the identity.
Moreover,  $F^\infty(g)$ is the sequential SOT-(resp. WOT-, w*-) closure of   $\cP({\bf W})$.
Similar results hold for  the noncommutative Hardy algebra  $R^\infty(g)$. In addition, we have  $R^\infty(g)=\{ W_1,\ldots, W_n\}^\prime=F^\infty(g)^\prime$.

An operator     $A:F^2(H_n)\otimes \cE\to F^2(H_n)\otimes \cF$   is called  multi-analytic operators  with respect to   the universal models 
$${\bf W}^{(f)}:=(W_1^{(f)},\ldots, W_n^{(f)})\quad \text{ and  } \quad  {\bf W}^{(g)}:=(W_1^{(g)},\ldots, W_n^{(g)})
$$
if 
  $$
A (W_i^{(g)}\otimes I_{\cE})=(W_i^{(f)}\otimes I_{\cF})A,\qquad  i\in \{1,\ldots, n\}.  
 $$

In what follows, we provide a commutant lifting theorem in our setting.

\begin{theorem} \label{commutant}
Let $f$ and $g$ be two admissible free holomorphic functions and let 
$${\bf W}^{(f)}:=(W_1^{(f)},\ldots, W_n^{(f)})\quad  \text{ and } \quad  {\bf W}^{(g)}:=(W_1^{(g)},\ldots, W_n^{(g)})
$$
 be the associated  universal models.
We assume that  $\cD_{g^{-1}}$ is a regular domain and $\cD^{pure}_{f^{-1}}\subset \cD^{pure}_{g^{-1}}$.
Let $\cG_1\subset F^2(H_n)\otimes \cE$ be a co-invariant subspace under  each operator 
$W_i^{(g) }\otimes I_\cE$ and let  $\cG_2\subset F^2(H_n)\otimes \cE_*$ be a co-invariant subspace under  each operator 
$W_i^{(f) }\otimes I_{\cE_*}$.  

If  $X\in B(\cG_1,\cG_2)$ is a contraction such that
$$
X[P_{\cG_1}(W_i^{(g)}\otimes I_{\cE})|_{\cG_1}]=[P_{\cG_2}(W_i^{(f)}\otimes I_{\cE_*})|_{\cG_2}]X,\qquad i\in \{1,\ldots, n\},
$$
then there is a contractive operator $\Gamma: F^2(H_n)\otimes \cE\to F^2(H_n)\otimes \cE_*$ such that
$$
\Gamma (W_i^{(g)}\otimes I_{\cE})=(W_i^{(f)}\otimes I_{\cE_*})\Gamma, \qquad  i\in \{1,\ldots, n\},
$$
 $X^*=\Gamma^*|_{\cG_2}$, and $\|\Gamma\|=\|X\|$. Moreover, there is a Hilbert space $\cF$ and a contractive multi-analytic operator $\Phi\in R^\infty(g)\bar\otimes_{min} B(\cE,\cF)$ and a co-isometric multi-analytic operator $F: F^2(H_n)\otimes \cF\to  F^2(H_n)\otimes \cE_*$ such that
$$
F (W_i^{(g)}\otimes I_{\cF})=(W_i^{(f)}\otimes I_{\cE_*})F, \qquad  i\in \{1,\ldots, n\},
$$
and $\Gamma=F\Phi$.
\end{theorem}
\begin{proof} 
Define $\tilde X:=(V|_{\cG_2})X:\cG_1\to V(\cG_2)$.  Due to Proposition  \ref{reduction}, we have 
$$ \tilde X[P_{\cG_1}(W_i^{(g)}\otimes I_{\cE})|_{\cG_1}]=P_{V(\cG_2)}(W_i^{(g)}\otimes I_{\cF})|_{V(\cG_2)} \tilde X,\qquad i\in \{1,\ldots, n\}.
$$
Since $\cD_{g^{-1}}$ is a regular domain, we can apply Theorem 4.2 from \cite{Po-domains} to find  a contractive operator
$\Phi\in R^\infty(g)\bar\otimes_{min} B(\cE,\cF)$  such that $\Phi^*|_{V(\cG_2)}=\tilde X^*$ and 
$\|\Phi\|=\|\tilde X\|=\|X\|$. 
Setting $\Gamma:=F\Phi$, where $F:=V^*$, we obtain 
$$
\Gamma^*|_{\cG_2}=\Phi^*V|_{\cG_2}=\tilde X^*V|_{\cG_2}=X^*P_{\cG_2}V^*|_{V(\cG_2)}V|_{\cG_2}=X^*
$$
The proof is complete.
\end{proof}

\begin{corollary} Let $f$ and $g$ be two admissible free holomorphic functions and let 
$${\bf W}^{(f)}:=(W_1^{(f)},\ldots, W_n^{(f)})\quad  \text{ and } \quad  {\bf W}^{(g)}:=(W_1^{(g)},\ldots, W_n^{(g)})
$$
 be the associated  universal models.
We assume that  $\cD_{g^{-1}}$ is a regular domain and $\cD^{pure}_{f^{-1}}\subset \cD^{pure}_{g^{-1}}$.

If $A:F^2(H_n)\otimes \cE\to F^2(H_n)\otimes \cE_*$ is a contractive operator  such that
$$
A (W_i^{(g)}\otimes I_{\cE})=(W_i^{(f)}\otimes I_{\cE_*})A, \qquad  i\in \{1,\ldots, n\},
$$
then A admits the factorization $A=F\Phi$ with the following properties.
\begin{enumerate}
\item[(i)]  $F: F^2(H_n)\otimes \cF\to  F^2(H_n)\otimes \cE_*$ is co-isometric and
$$
F (W_i^{(g)}\otimes I_{\cF})=(W_i^{(f)}\otimes I_{\cE_*})F, \qquad  i\in \{1,\ldots, n\}.
$$
\item[(ii)]  $\Phi\in R^\infty(g)\bar\otimes_{min} B(\cE,\cF)$  is a contractive multi-analytic operator. 
\end{enumerate}
Conversely, if $F,\Phi$ satisfy  conditions (i) and (ii), then $A:=F\Phi$ is a contractive operator   such that 
$$
A (W_i^{(g)}\otimes I_{\cE})=(W_i^{(f)}\otimes I_{\cE_*})A, \qquad  i\in \{1,\ldots, n\}.
$$
\end{corollary}
\begin{proof} Take $\cG_1=F^2(H_n)\otimes \cE$ and $\cG_2=F^2(H_n)\otimes \cE_*$ in  Theorem 
\ref{commutant}.
\end{proof}

As applications of the commutant lifting  of Theorem \ref{commutant}, we provide  some factorization results for multi-analytic operators  and   a Toeplitz-corona theorem in our setting.

\begin{theorem} \label{corona1}
Let $f_1,f_2$ and $g$ be three admissible free holomorphic functions and let 
${\bf W}^{(f_1)}$,    ${\bf W}^{(f_2)}$, and ${\bf W}^{(g)}$ be the associated universal models, respectively.
  We assume that  $\cD_{g^{-1}}$ is a regular domain and $\cD^{pure}_{f_1^{-1}}\cup\cD^{pure}_{f_2^{-1}} \subset \cD^{pure}_{g^{-1}}$.
 
 Let $A:F^2(H_n)\otimes \cF\to F^2(H_n)\otimes \cF_1$   and  
 $B:F^2(H_n)\otimes \cF_2\to F^2(H_n)\otimes \cF_1$  be multi-analytic operators, i.e. 
  $$
A (W_i^{(g)}\otimes I_{\cF})=(W_i^{(f_1)}\otimes I_{\cF_1})A \ \text{ and } \  B (W_i^{(f_2)}\otimes I_{\cF_2})=(W_i^{(f_1)}\otimes I_{\cF_1})B
$$
for every $ i\in \{1,\ldots, n\}$. Then there is a contractive  operator $C:F^2(H_n)\otimes \cF\to F^2(H_n)\otimes \cF_2$ such that 
$$
C (W_i^{(g)}\otimes I_{\cF})=(W_i^{(f_2)}\otimes I_{\cF_2})C,\qquad  i\in \{1,\ldots, n\},
$$
 and $A=BC $ if and only if $AA^*\leq BB^*$. 
  \end{theorem}
\begin{proof} The direct implication is obvious. To prove the converse, assume that  $AA^*\leq BB^*$.
Let $\cM:=\overline{B^*(F^2(H_n)\otimes \cF_1)}$ and define the map $X:\cM\to F^2(H_n)\otimes \cF$  by setting
$XB^*x:=A^*x$ for $x\in F^2(H_n)\otimes \cF$. 
 Since $\|A^*x\|\leq \|B^*x\|$, the operator $X$ is well-defined and has a unique extension, also denoted by $X$, to a bounded operator on $\cM$. Thus $XB^*=A^*$.
Since  
$B (W_i^{(f_2)}\otimes I_{\cF_2})=(W_i^{(f_1)}\otimes I_{\cF_1})B$
for every $ i\in \{1,\ldots, n\}$,  the subspace $\cM$ is invariant under   each operator $W_i^{(f_2) *}\otimes I_{\cF_2}$. Due to the analyticity  of $A$ and $B$ and using relation $XB^*=A^*$, we obtain
\begin{equation*}
\begin{split}
X(W_i^{(f_2) *}\otimes I_{\cF_2})B^*&=XB^*(W_i^{(f_1)*}\otimes I_{\cF_1})=A^*(W_i^{(f_1)*}\otimes I_{\cF_1})\\
&= (W_i^{(g)*}\otimes I_{\cF}) A^*= (W_i^{(g)*}\otimes I_{\cF})XB^*.
\end{split}
\end{equation*}
 Therefore,  $X(W_i^{(f_2) *}\otimes I_{\cF_2})|_\cM=(W_i^{(g)*}\otimes I_{\cF})$ for every $ i\in \{1,\ldots, n\}$, which  implies 
 $$
 \left[P_\cM(W_i^{(f_2)}\otimes I_{\cF_2})|_\cM\right] X^*=X^*(W_i^{(g)}\otimes I_{\cF}),\qquad i\in \{1,\ldots, n\}.
 $$
 Applying the commutant lifting  of Theorem  \ref{commutant}, we find  a contraction 
 $C:F^2(H_n)\otimes \cF\to F^2(H_n)\otimes \cF_2$  such that
$$
C (W_i^{(g)}\otimes I_{\cF})=(W_i^{(f_2)}\otimes I_{\cF_2})C, \qquad  i\in \{1,\ldots, n\},
$$
$\|C\|=\|X\|$, and $C^*|_\cM=X$. Since  $\cM:=\overline{B^*(F^2(H_n)\otimes \cF_1)}$,  we deduce that
$
A=BX^*=BP_\cM C=BC.
$
 The proof is complete.
 \end{proof}

\begin{corollary}  \label{multi2} Let $E:F^2(H_n)\otimes \cF\to F^2(H_n)\otimes \cF_1$ be an operator such that
$$
E (W_i^{(g)}\otimes I_{\cF})=(W_i^{(f_1)}\otimes I_{\cF_1})E,\qquad i\in \{1,\ldots, n\}.
$$ 
Then $E$ admits a right inverse if and only if  there is  an operator
$F:F^2(H_n)\otimes \cF_1\to F^2(H_n)\otimes \cF$  such that
$$
F (W_i^{(f_1)}\otimes I_{\cF_1})=(W_i^{(g)}\otimes I_{\cF})F,\qquad  i\in \{1,\ldots, n\}.
$$ 
and $EF=I$.
\end{corollary}
\begin{proof} We only have to prove the direct implication. If $A$ has a right inverse then  there is $\delta>0$ such that  $\|E^*x\|\geq \|G^*x\|$ for  $x\in  F^2(H_n)\otimes \cF_1$, where $G:=\sqrt{\delta} I$. Applying Theorem \ref{corona1}, we find a contraction $D:F^2(H_n)\otimes \cF_1\to F^2(H_n)\otimes \cF$ such that
$$
D (W_i^{(f_1)}\otimes I_{\cF_1})=(W_i^{(g)}\otimes I_{\cF})D,\qquad  i\in \{1,\ldots, n\},
$$ 
and $G=ED$. Setting $F:=\frac{1}{\sqrt{\delta}}D$, we obtain $EF=I$ and $\|F\|\leq \frac{1}{\sqrt{\delta}}$.
\end{proof}

\begin{theorem} \label{corona2} Let $f_1$ and $g$ be two admissible free holomorphic functions and let 
${\bf W}^{(f_1)}$ and ${\bf W}^{(g)}$ be the associated universal models, respectively.
  We assume that  $\cD_{g}$ is a regular domain and $\cD^{pure}_{f_1^{-1}}  \subset \cD^{pure}_{g}$.
 Let $\Phi_k:F^2(H_n)\otimes \cF\to F^2(H_n)\otimes \cF_1$, $k\in\{1,\ldots m\}$,  be  multi-analytic operators such that
$$
\Phi_k (W_i^{(g)}\otimes I_{\cF})=(W_i^{(f_1)}\otimes I_{\cF_1})\Phi_k,\qquad i\in \{1,\ldots, n\}, k\in\{1,\ldots m\}.
$$ 
Then the following statements are equivalent.
\begin{enumerate}
\item[(i)]  There exist operators  $\Psi_k:F^2(H_n)\otimes \cF_1\to F^2(H_n)\otimes \cF$, $k\in\{1,\ldots m\}$,   such that
$$
\Psi_k (W_i^{(f_1)}\otimes I_{\cF_1})=(W_i^{(g)}\otimes I_{\cF})\Psi_k,\qquad i\in \{1,\ldots, n\},
$$ 
and 
$$
\Phi_1\Psi_1+\cdots + \Phi_m\Psi_m=I
$$
\item[(ii)] There exists $\delta>0$ such that
$$
\|\Phi_1^*x\|^2+\cdots +\|\Phi_m^*x\|^2\geq \delta \|x\|^2,\qquad x\in F^2(H_n)\otimes \cF_1.
$$
\end{enumerate}
\end{theorem}
\begin{proof} Define $\Phi:=[\Phi_1 \cdots  \Phi_k]$ and $\Psi:=[\Psi_1 \cdots  \Psi_k]^t$, the transpose,  and note that
$$
\Phi(W_i^{(g)}\otimes I_{\oplus_1^k \cF})= (W_i^{(f_1)}\otimes I_{\cF_1}) \Phi\quad \text{and}\quad 
\Psi(W_i^{(f_1)}\otimes I_{\cF_1})= (W_i^{(g)}\otimes I_{\oplus_1^k\cF}) \Phi
$$
for every $i\in \{1,\ldots, n\}$. Since (i)  is equivalent to relation  $\Phi\Psi=I$, an application of Corollary \ref{multi2} completes the proof.
\end{proof}

\bigskip

 \section{Boundary property for universal models}
 
 In this section, we provide a large class of universal models ${\bf W}=(W_1,\ldots, W_n)$ with the boundary property. In this case, we show that the full Fock space $F^2(H_n)$ is rigid with respect to  ${\bf W}$.
 
 We recall some basics concerning the Arveson's  theory of boundary representations \cite{Arv1-acta}, \cite{Arv2-acta}.
 Let $\cB$ be a unital $C^*$-algebra and let $\cS$ be a linear subspace of $\cB$, which contains the identity of $\cB$, and such that $\cB=C^*(\cS)$. An irreducible representation $\pi:C^*(\cS)\to B(\cH)$ is said to be a {\it boundary  representation} for  $\cS$ if $\pi|_\cS$ has a unique completely positive extension to $\cB$, namely, $\pi$ itself. Arveson's boundary theorem \cite{Arv2-acta} states that if $\cS\subset B(\cH)$ is an irreducible linear subspace such that $I\in \cS$  and the $C^*$-algebra $C^*(\cS)$ contains the compact operators on $\cH$, then the identity representation  of 
 $C^*(\cS)$ is a boundary representation for $\cS$ if and only if the quotient map
 $Q:B(\cH)\to B(\cH)/{\cK(\cH)}$ is not completely isometric  on the linear span of $\cS\cup \cS^*$.

  Let $g= 1+\sum_{ |\alpha|\geq 1} b_\alpha
Z_\alpha$ be a free holomorphic function in a neighborhood of the origin such that  $b_\alpha>0$ and 
$$
\sup_{\alpha\in \FF_n^+} \frac{b_\alpha}{b_{g_i \alpha}}<\infty,\quad \text{ for every }\ i\in \{1,\ldots, n\},
$$ 
and 
let   ${\bf W}:=(W_1,\ldots, W_n)$ be  the   weighted left creation operators associated with $g$.
 We recall that, due to Theorem \ref{irreducible}, the $C^*$-algebra  $C^*({\bf W})$ is irreducible.
    
 \begin{definition}  We say that ${\bf W}$ has the boundary property if the identity representation of the  $C^*$-algebra $C^*({\bf W})$ is a boundary representation for the operator space $\Span \{I,W_1,\ldots, W_n\}$.
 \end{definition}

The next result provides a large class  of universal models ${\bf W}=(W_1,\ldots, W_n)$ with  the boundary property.

\begin{theorem}
 Let $g= 1+\sum_{ |\alpha|\geq 1} b_\alpha
Z_\alpha$ be a free holomorphic function in a neighborhood of the origin such that  $b_\alpha>0$ and 
$$
\sup_{\alpha\in \FF_n^+} \frac{b_\alpha}{b_{g_i \alpha}}<\infty,\quad \text{ for every }\ i\in \{1,\ldots, n\},
$$ 
and 
let   ${\bf W}:=(W_1,\ldots, W_n)$ be  the   weighted left creation operators associated with $g$.
 Assume that the following conditions hold.
\begin{enumerate}
\item[(i)] There is $i\in \{1,\ldots, n\}$ such that 
$$
\lim_{\gamma\to \infty} \left(\frac{b_{g_p\gamma}}{b_{g_ig_p\gamma}}-\frac{b_\gamma}{b_{g_p\gamma}}\right) =0,\qquad p\in \{1,\ldots, n\}.
$$
\item[(ii)] The following  limit  exists  and \ $\lim_{\gamma\to \infty}\sum_{i=1}^n \frac{b_\gamma}{b_{g_i\gamma}}<\sup_{\gamma\in \FF_n^+}\sum_{i=1}^n \frac{b_\gamma}{b_{g_i\gamma}}$. 
\end{enumerate}
Then the  $n$-tuple ${\bf W}:=(W_1,\ldots, W_n)$   has the boundary property.
\end{theorem}
\begin{proof} According to Theorem \ref{irreducible}, if item (i) holds, then 
   the $C^*$-algebra $C^*({\bf W})$ contains all the compact operators in $B(F^2(H_n))$.  
   Using the definition of the operators $W_1,\ldots, W_n$, we deduce that
   \begin{equation}
   \label{sumWW}
   \sum_{i=1}^n W_i^*W_i=\sum_{\gamma\in \FF_n^+}\left(\sum_{i=1}^n \frac{b_\gamma}{b_{g_i\gamma}}\right)P_{\CC e_\gamma},
   \end{equation}
where  $P_{\CC e_\gamma}$ is the orthogonal projection onto the one-dimensional  space  $\CC e_\gamma$.
Setting $a:=\lim_{\gamma\to \infty}\sum_{i=1}^n \frac{b_\gamma}{b_{g_i\gamma}}$, we deduce that
 $$
   \sum_{i=1}^n W_i^*W_i -aI=\sum_{\gamma\in \FF_n^+}\left(\sum_{i=1}^n \frac{b_\gamma}{b_{g_i\gamma}}-a\right)P_{\CC e_\gamma}.
   $$
Note that $ \sum_{i=1}^n W_i^*W_i -aI$ is a compact operator  and, consequently, the essential norm satisfies the inequality
$$
\left\| \sum_{i=1}^n W_i^*W_i \right\|_e\leq a.
$$
On  the other hand, using relation  \eqref{sumWW}, we deduce that
$$
\left\| \sum_{i=1}^n W_i^*W_i \right\|=\sup_{\gamma\in \FF_n^+}\sum_{i=1}^n \frac{b_\gamma}{b_{g_i\gamma}}.
$$
Now, due to item(ii), we obtain
$$
\left\| \sum_{i=1}^n W_i^*W_i \right\|_e<\left\| \sum_{i=1}^n W_i^*W_i \right\|.
$$
Therefore, the Calkin map
 $Q:B(\cH)\to B(\cH)/{\cK(\cH)}$ is not completely isometric  on the linear span $\Span \{I,W_1,\ldots, W_n, W_1^*,\ldots, W_n^*\}$. Indeed, taking 
 $$
 M=\left(\begin{matrix} W_1&0&\cdots&0\\ \vdots&\vdots&\vdots&\vdots\\  W_n&0&\cdots&0
 \end{matrix}\right),
$$
we have $\|M\|^2_e=\|M^*M\|_e<\|M^*M\|=\|M\|^2$.
Applying now Arveson's boundary theorem, we conclude that  the identity representation of the  $C^*$-algebra $C^*({\bf W})$ is a boundary representation for the operator space $\Span \{I,W_1,\ldots, W_n\}$.
Hence, the  $n$-tuple ${\bf W}:=(W_1,\ldots, W_n)$   has the boundary property.
The proof is complete.
\end{proof}

\begin{corollary} Let $g= 1+\sum_{ |\alpha|\geq 1} b_\alpha
Z_\alpha$ be a free holomorphic function in a neighborhood of the origin such that  $b_\alpha>0$, 
$$
\sup_{\alpha\in \FF_n^+} \frac{b_\alpha}{b_{g_i \alpha}}<\infty,\quad \text{ for every }\ i\in \{1,\ldots, n\},
$$ 
and 
there is $i\in \{1,\ldots, n\}$ such that 
$$
\lim_{\gamma\to \infty} \left(\frac{b_{g_p\gamma}}{b_{g_ig_p\gamma}}-\frac{b_\gamma}{b_{g_p\gamma}}\right) =0,\qquad p\in \{1,\ldots, n\}. 
$$
If   ${\bf W}:=(W_1,\ldots, W_n)$  are  the   weighted left creation operators associated with $g$,
  then
$$
\left\| \sum_{i=1}^n W_i^*W_i \right\|=\sup_{\gamma\in \FF_n^+}\sum_{i=1}^n \frac{b_\gamma}{b_{g_i\gamma}}
\quad \text{and} \quad 
\left\| \sum_{i=1}^n W_i^*W_i \right\|_e\leq \lim_{\gamma\to \infty}\sum_{i=1}^n \frac{b_\gamma}{b_{g_i\gamma}}.
$$
\end{corollary}

\begin{corollary}  \label{ab}  Let $g= 1+\sum_{ |\alpha|\geq 1} b_\alpha
Z_\alpha$ be a free holomorphic function in a neighborhood of the origin  
  with $b_\alpha=b_\beta>0$ if  $|\alpha|=|\beta|$  such that the following limit exists and 
$$
\lim_{k\to \infty} \frac{b_k}{b_{k+1}}=1< \sup_{k\in \NN\cup\{0\}}\frac{b_k}{b_{k+1}}<\infty,
$$
where $b_k:=b_\alpha$ for any $\alpha\in  F^2(H_n)$  with $|\alpha|=k\in \NN$.
Then the $n$-tuple  ${\bf W}:=(W_1,\ldots, W_n)$ associated with $g$ has the boundary property.
\end{corollary}

\begin{example}   Let $s\in \RR$ and let
$$
\xi_s=1+\sum_{\alpha\in \FF_n^+, |\alpha|\geq 1} (|\alpha|+1)^s Z_\alpha.
$$
Then $\lim_{k\to \infty} \frac{b_k}{b_{k+1}}=1$ and 
$$
\sup_{k\in \NN\cup\{0\}}\frac{b_k}{b_{k+1}}=\begin{cases}\frac{1}{2^s},&\quad  s<0\\
1,&\quad s\geq 0.\end{cases}
$$
Applying  Corollary \ref{ab}, we deduce that,
if $s<0$,
then the universal model ${\bf W}:=(W_1,\ldots, W_n)$  associated with $\xi_s$ has the boundary property. In particular, if $s=-1$,  the universal model of the noncommutative Dirichlet domain   has the boundary property. 
\end{example}

\begin{example}
For each $s\in (0,\infty)$, let 
$$
g_s=1+\sum_{k=1}^\infty\left(\begin{matrix} s+k-1 \\k \end{matrix}\right)(Z_1+\cdots +Z_n)^k.
$$ 
In this case, we have 
$\frac{b_k}{b_{k+1}}=\frac{k+1}{s+k}$. Consequently, $\lim_{k\to \infty} \frac{b_k}{b_{k+1}}=1$ and 
$$
\sup_{k\in \NN\cup\{0\}}\frac{b_k}{b_{k+1}}=\begin{cases}\frac{1}{s},&\quad  s\leq 1\\
1,&\quad s >1.\end{cases}
$$
Applying  Corollary \ref{ab}, we deduce that,
 if $s<1$,
then the universal model ${\bf W}:=(W_1,\ldots, W_n)$  associated with $g_s$  has the boundary property.  Due to Corollary \ref{RR},  ${\bf W}$ does not have the boundary property when  $s=1$. In this case, ${\bf W}=(S_1,\ldots, S_n)$, the $n$-tuple of left creation operators on the full Fock space.
\end{example}

Using the idea of the proof  of Proposition 8.1 from \cite{Arv3-acta}, we prove the following.

\begin{proposition} \label{Arveson}
Assume that  the universal model ${\bf W}:=(W_1,\ldots, W_n)$ has the boundary property. If $\{A_k\}_{k=1}^N$  is a sequence  ($N\in \NN$ or $N=\infty$) of operators in the commutant of the set $\{W_1,\ldots, W_n\}$ such that 
$$
A_1^* A_1+A_2^*A_2+\cdots =I,
$$
then $A_k=c_k I$  for some $c_k\in \CC$ such that $\sum_{k=1}^\infty |c_k|^2=1$.
\end{proposition}
\begin{proof}
Define  the map $\Omega:C^*({\bf W})\to B(F^2(H_n))$ by setting
 $\Omega(X):=\sum_{k=1}^N A_i^*XA_i$. When $N=\infty$, the convergence  of the series is in the strong operator topology. Due to the fact that  $\{A_k\}_{k=1}^N\subset \{W_1,\ldots, W_n\}'$, we have  $\Omega(I)=I$, $\Omega(W_i)=W_i$ and  $\Omega(W_i^*)=W_i^*$ for every $i\in \{1,\ldots, n\}$.
Since ${\bf W}:=(W_1,\ldots, W_n)$ has the boundary property, we must have $\Omega(X)=X$ for every $X\in C^*({\bf W})$. 
Define $U:=\left[\begin{matrix} A_1\\A_2\\ \vdots \end{matrix}\right]:F^2(H_n)\to \bigoplus_1^N F^2(H_n)$ and note that $U$ is an isometry.
Consider the representation  $\pi:C^*({\bf W})\to \bigoplus_1^N F^2(H_n)$ defined by $\pi(X):=\oplus_1^N X$.
Note that  $\Omega(X)=U^*\pi(X)U$ for all $X\in C^*({\bf W})$ and
\begin{equation*}
\begin{split}
(UX-\pi(X)U)^*(UX-\pi(X)U)&=X^*U^*UX-U^*\pi(X^*)UX-X^*U^*\pi(X)U+U^*\pi(X^*X)U\\
&=XX^*-\Omega(X^*)X-X^*\Omega(X)+\Omega(X^*X)=0
\end{split}
\end{equation*}
for all $X\in C^*({\bf W})$. Consequently, $UX=\pi(X)U$, which shows that $A_kX=XA_k$ for every $X\in C^*({\bf W})$ and $k\in \{1,\ldots, N\}$.
Since the $C^*$-algebra  $C^*({\bf W})$ is irreducible (see Theorem \ref{irreducible}), we deduce that $A_k=c_k I$  for some $c_k\in \CC$ such that $\sum_{k=1}^\infty |c_k|^2=1$.
The proof is complete.
\end{proof}

\begin{corollary} \label{RR} The $n$-tuple of left creation operators $(S_1,\ldots, S_n)$  on the full Fock space $F^2(H_n)$  does not have the boundary property.

\end{corollary}
\begin{proof} Note that the right creation operators  $R_1,\ldots, R_n$ are in the commutant of $\{S_1,\ldots, S_n\}$. Since $R_i^*R_i=I$, we can apply Proposition \ref{Arveson}, to conclude that 
$(S_1,\ldots, S_n)$  does not have the boundary property.
\end{proof}

Let $\cM\subset F^2(H_n)$ and $\cN\subset F^2(H_n)$ be  closed invariant  invariant subspaces under $W_1,\ldots, W_n$. We say that $\cM$ and $\cN$ are equivalent if there is a unitary operator $U:\cM\to \cN$ such that
$
UW_i=W_i U, i\in \{1,\ldots, n\}.
$
We say  that  an invariant subspace $\cM\subset F^2(H_n)$  under $W_1,\ldots, W_n$ is rigid  (or that $(\cM, {\bf W})$ is  rigid) if any invariant subspace $\cN\subset F^2(H_n)$ which is unitarily equivalent to $\cM$ must  be $\cM$ itself.  

\begin{proposition}
Assume that the universal model ${\bf W}:=(W_1,\ldots, W_n)$ associated with $g$  has the boundary property. Then $(F^2(H_n), {\bf W})$ is rigid.
\end{proposition}
\begin{proof}
Let $\cM\subset F^2(H_n)$ be  an invariant  invariant subspaces under $W_1,\ldots, W_n$ and let $U:F^2(H_n)\to \cM$ be a unitary operator such that  $UW_i=W_i U$ for every 
$ i\in \{1,\ldots, n\}.$ Since $U^*U=I$, we can apply Proposition  \ref{Arveson} to deduce that $U=cI$ for some $c\in \CC\backslash\{0\}$. Consequently, $\cM=F^2(H_n)$, which completes the proof.
\end{proof}

      \bigskip

      %\Refs
      %\widestnumber\key{BFPQR}
      %\def\n{\key}
       %

      \end{document}